\documentclass[12pt]{amsart}
\usepackage{times}
\usepackage{anysize}
\marginsize{2.7cm}{2.44cm}{2.0cm}{3.0cm}
\usepackage[small,it]{caption}
 \usepackage[latin1]{inputenc}
 \usepackage[dvips]{graphicx}
 \usepackage{wrapfig}
 \usepackage{amsmath}
 \usepackage{amsthm}
 \usepackage{amsfonts}
 \usepackage{amssymb}
 \usepackage{layout}
 \usepackage{verbatim}
 \usepackage{alltt}
\usepackage{yfonts}

\usepackage[all]{xy}
\usepackage{setspace}

\newtheorem*{thma}{Theorem A}
\newtheorem*{thmb}{Theorem B}
\newtheorem*{thmc}{Theorem C}
\newtheorem*{claim}{Claim}

\newcommand{\LL}{\Lambda}
\newcommand{\QQ}{\mathbb{Q}}
\newcommand{\FF}{\mathcal{F}}

\newcommand{\lra}{\longrightarrow}
\newcommand{\ZZ}{\mathbb{Z}}
\newcommand{\PP}{\mathcal{P}}
\newcommand{\Gal}{\textup{Gal}}
\newcommand{\KS}{\textbf{\textup{KS}}}
\newcommand{\ES}{\textbf{\textup{ES}}}
\newcommand{\NN}{\mathcal{N}}
\newcommand{\ra}{\rightarrow}
\newcommand{\xx}{\mathbf{X}}
\newcommand{\be}{\begin{equation}}
\newcommand{\ee}{\end{equation}}

\newcommand{\XX}{\mathcal{X}}
\newcommand{\kk}{\mathcal{K}}
\newcommand{\al}{\mathcal{L}}

\newcommand{\qq}{\textup{\frakfamily q}}

\newcommand{\ff}{\hbox{\frakfamily f}}
\newcommand{\all}{\mathbb{L}}
\newcommand{\FFc}{\mathcal{F}_{\textup{\lowercase{can}}}}
\newcommand{\Hom}{\textup{Hom}}
\newcommand{\cl}{\textup{cl}}
\newcommand{\hone}{$\mathbf{H.1}$}

\newcommand{\hthree}{$\mathbf{H.3}$}
\newcommand{\hfour}{$\mathbf{H.4}$}
\newcommand{\hsezF}{$\mathbf{\mathbb{H}.sEZ}_{/k}$}
\newcommand{\htamF}{$\mathbf{\mathbb{H}.T}_{/k}$}
\newcommand{\htam}{$\mathbf{\mathbb{H}.T}$}
\newcommand{\hsez}{$\mathbf{\mathbb{H}.sEZ}$}
\newcommand{\finite}{\textup{f}}
\newcommand{\unr}{\textup{unr}}

\numberwithin{equation}{section}
\newtheorem{thm}{Theorem}[section]
\newtheorem{lemma}[thm]{Lemma}
\newenvironment{define}{\par\medskip\noindent\refstepcounter{thm}
\bgroup{\hspace*{-0.15 cm}\bf{Definition}
\thethm.}\bgroup}{\egroup \egroup\par\medskip}\newtheorem{prop}[thm]{Proposition}
\newtheorem{cor}[thm]{Corollary}
\newenvironment{rem}{\par\medskip\noindent\refstepcounter{thm}
\bgroup{\hspace*{-0.15 cm}\bf{Remark} \thethm.}\bgroup}{\egroup
\egroup\par\medskip}
\newenvironment{rmkdef}{\par\medskip\noindent\refstepcounter{thm}
\bgroup{\hspace*{-0.15 cm}\bf{Remark/Definition} \thethm.}\bgroup}{\egroup
\egroup\par\medskip}
\newenvironment{example}{\par\medskip\noindent\refstepcounter{thm}
\bgroup{\hspace*{-0.15 cm}\bf{Example}
\thethm.}\bgroup}{\egroup \egroup\par\medskip}
\newenvironment{assume}{\par\medskip\noindent\refstepcounter{thm}
\bgroup{\hspace*{-0.15 cm}\bf{Assumption}
\thethm.}\bgroup}{\egroup \egroup\par\medskip}

\begin{document}
\title{{S}\lowercase{tickelberger elements and} {K}\lowercase{olyvagin systems}}

\author{K\^az\i m B\"uy\"ukboduk}

\address{IH\'ES, Le Bois-Marie, 35,  \hfill\break\indent Route de Chartres \hfill\break\indent F-91440 Bures-sur-Yvette
\hfill\break\indent France}

\curraddr{
\hfill\break\indent Ko\c{c} University, Mathematics \hfill\break\indent Rumeli Feneri Yolu \hfill\break\indent 34450 Sar\i yer / \.Istanbul
\hfill\break\indent Turkey}



\begin{abstract}

In this paper, we construct (many) Kolyvagin systems out of
Stickelberger elements utilizing ideas borrowed from our previous
work on Kolyvagin systems of Rubin-Stark elements. The applications
of our approach are two-fold: First, \emph{assuming Brumer's
conjecture}, we prove results on the \emph{odd} parts of the ideal
class groups of CM fields which are abelian over a totally real
field, and deduce Iwasawa's main conjecture for totally real fields
(for totally odd characters). Although this portion of our results
have already been established by Wiles unconditionally (and refined
by Kurihara using an Euler system argument, when Wiles' work is
assumed), the approach here fits well in the general framework the
author has developed elsewhere to understand Euler / Kolyvagin
system machinery when the \emph{core Selmer rank} is $r>1$ (in  the
sense of Mazur and Rubin). As our second application, we establish a
rather curious link between the Stickelberger elements and
Rubin-Stark elements by using the main constructions of this article
hand in hand with the `rigidity' of the collection of Kolyvagin
systems proved by Mazur, Rubin and the author.

\end{abstract}

\maketitle
\tableofcontents
\section{Introduction}
\label{introduction}
The Euler / Kolyvagin system machinery is designed to bound the size of a Selmer group. In all well-known cases the bounds obtained relate to $L$-values, and thus provide a link between arithmetic and analytic data. Well-known prototypes for such a relation between arithmetic and analytic data are the Birch and Swinnerton-Dyer conjecture (more generally, Bloch-Kato conjectures) and the main conjectures of Iwasawa theory. The Kolyvagin system machinery has been successfully applied by  many to obtain deep results towards proving these conjectures.

B. Howard, B. Mazur and K. Rubin show in~\cite{mr02} that the
existence of Kolyvagin systems relies on a cohomological invariant,
what they call the \emph{core Selmer rank} (c.f., \cite[Definition
4.1.11]{mr02}). When the core Selmer rank is one, they determine the
structure of the Selmer group completely in terms of a Kolyvagin
system. However, when the core Selmer  rank is greater than one, not
much could be said. One of the principal objectives of the current
article is to explore this more mysterious case in depth for a
particular Galois representation, for which the core Selmer rank is
greater than 1.


Let $k$ be a totally real field of degree $[k:\QQ]=r$ and let
$G_k:=\Gal(\overline{k}/k)$ be its absolute Galois group. Fix once
and for all an odd rational prime $p$, and let
$\psi:G_k\ra\ZZ_p^\times$ be a totally even character which has
finite prime-to-$p$ order. Consider the $G_k$-representation
$T^\prime:=\ZZ_p(1)\otimes\psi^{-1}$. The core Selmer rank of the
Galois representation $T^\prime$ is $r$, and $T^\prime$ leads us to
one of the basic instances when the core Selmer rank is greater than
one. In~\cite{kbbstark, kbbiwasawa}, the author has studied the
Kolyvagin system machinery for $T^\prime$. The main idea in these
two papers is  to modify the relevant  Selmer group appropriately
and construct Kolyvagin systems out of Rubin-Stark elements defined
in~\cite{ru96} so as to control this modified Selmer group. In this
paper, we consider the $G_k$-representation $T=\ZZ_p(\chi)$, where
$\chi:G_k \ra \ZZ_p^\times$ is a totally odd character which has
finite prime-to-$p$ order. The $G_k$-representation $T$ turns out to
have core Selmer rank $r$ as well. As in~\cite{kbbstark,
kbbiwasawa}, we introduce certain modified Selmer groups associated
with the representation $T$. In this setting, the Euler system that
gives rise to the Kolyvagin system which controls the modified
Selmer group is obtained from the Stickelberger elements.

 Before we state the main results of this article, we set some notation which will be in effect throughout the paper. Let $p, k, G_k$ and $r$ be as above, and let $\chi:G_k\ra\ZZ_p^\times$ be a totally odd character, which is different from the Teichm\"uller character $\omega$ that gives the action of $G_k$ on the $p$-th roots of unity. Let $k_\infty$ denote the cyclotomic $\ZZ_p$-extension of $k$. Consider the following properties:
 \begin{enumerate}
 \item[(A1)] For any prime $\wp$ of $k$ above $p$, we have $\chi(\wp)\neq 1$.
  \item[(A2)] Any prime of $k$ above $p$ totally ramifies in $k_\infty/k$.
 \end{enumerate}
 Note that (A2) is true, for example, if $k/\QQ$ is unramified. The
 hypothesis (A1) ensures that the associated Deligne-Ribet $p$-adic
 $L$-function $\mathcal{L}_{\omega\chi^{-1}}$ does not have a
 trivial zero in the sense of~\cite{gr2}. Note that Wiles gave a
 proof of the main conjecture without assuming these two hypotheses.

Let $L$ be the fixed field of $\ker(\chi)$ inside a fixed algebraic closure $\overline{k}$ of $k$, write $\Delta=\Gal(L/k)$. For any number field $K$ containing $L$, let $A_K$ be the $p$-part of the ideal class group of $K$, and $A_K^\chi$ its $\chi$-isotypic part.  We fix $S$ as the set of places of $k$ which consists of  all infinite places of $k$, all places $\lambda$ which divide the conductor $\ff_\chi$ of $\chi$, as well as all the places of $k$ above $p$.  Finally, let $\theta_{L,S}=\theta_L \in \ZZ_p[\Delta]$ be the \emph{Stickelberger element} (defined precisely in \cite[\S1.2]{kurihara}, see also \S\ref{sec:ES} below) relative to $S$.
For the main results of this article, we will assume the $\chi$-part of  Brumer's conjecture:
\begin{assume}
\label{assume:brumer}
$\theta_K^{\chi}$ annihilates $A_K^\chi$.
\end{assume}

We remark here that Wiles~\cite{wiles90} proved that Brumer's
conjecture as stated above follows from his
proof~\cite{wiles-mainconj} of the \emph{main conjecture of Iwasawa
theory} for totally real fields. In this paper, we prove the other
way around, namely that, assuming Brumer's conjecture, one might
also prove the main conjecture (see Theorem B below; see also
Kurihara's work~\cite{kurihara} where he refines Wiles' result using
a different type of Euler system argument).

The first application of the treatment here is the following (Theorem~\ref{thm:mainequalityk} below):
 \begin{thma}
\label{thm:mainkA}
Suppose the hypothesis \textup{(A1)} and Assumption~\ref{assume:brumer} hold. Then,
$$|A_L^\chi|=|\ZZ_p/\chi(\theta_L)\ZZ_p|.$$
\end{thma}

With a bit more work, we can prove the Iwasawa theoretic version of Theorem A, which we state below. Set $\Gamma=\Gal(k_\infty/k)$ and $\LL=\ZZ_p[[\Gamma]]$, as usual. Let $\al_{\omega\chi^{-1}}\in \LL$ denote the Deligne-Ribet $p$-adic $L$-function (see~\cite{deligne-ribet}). We  recall in~(\ref{eqn:interpolation}) the basic interpolation property which characterizes $\al_{\omega\chi^{-1}}$. Let $\textup{Tw}_{\langle\rho_{\textup{cyc}}\rangle}$ be a certain twisting operator on $\LL$ (see \S\ref{subsec:mainthmkinfty} below for its definition). For any abelian group $A$, let $A^{\vee}:=\Hom(A,\QQ_p/\ZZ_p)$ denote its Pontryagin dual, and finally, let $\textup{char}(M)$ denote the characteristic ideal of a finitely generated $\LL$-module $M$ (with the convention that $\textup{char}(M)=0$ unless $M$ is $\LL$-torsion). Then we are able to prove (see Theorem~\ref{thm:mainequalitykinfty} and Corollary~\ref{cor:mainequalityinfty} for a slightly improved version so as to include the case $\pmb{\mu}_p \subset L$):

 \begin{thmb}
Suppose the hypotheses \textup{(A1)}-\textup{(A2)} as well as Assumption~\ref{assume:brumer} hold. Assume also that $L$ does not contain $p$th roots of unity. Then,
$$\textup{char}\left((\varinjlim_n A_{L_n}^\chi)^{\vee}\right)=\textup{Tw}_{\langle\rho_{\textup{cyc}}\rangle}(\al_{\omega\chi^{-1}}).$$
\end{thmb}
 These results have already been obtained by Wiles~\cite{wiles-mainconj} without appealing to the Euler system machinery, even without the assumptions of Theorem B above. Kurihara~\cite{kurihara} proved a refined version of Theorem B using an Euler system argument (still building on Wiles' results), though significantly different than ours. The novelty in this paper is a new treatment of the Euler / Kolyvagin system machinery \emph{\`a la} what we call \emph{$\pmb{\al}$-restricted Euler systems} (see \S\ref{subsec:choosehoms}, especially Definition~\ref{rmkdef:XrestrictedES} and Example~\ref{ex:Lrestricted}), when the \emph{core Selmer rank} (in the sense of Mazur and Rubin~\cite{mr02}) of the Galois representation 
 in question is $r>1$. One benefit of this new approach presented in this paper is a rather surprising link between Stickelberger elements and the (conjectural) Rubin-Stark elements, which we prove below in Theorem~\ref{thm:maincomparison}; see also Theorem C in this section below and the paragraph that follows its statement.

 As in our earlier papers~\cite{kbbstark, kbbiwasawa}, we improve the Euler system / Kolyvagin system machinery of~\cite{r00, pr-es, kato-es, mr02} to prove Theorem A and B, generalizing Rubin's treatment~\cite[\S3.4]{r00} of the Stickelberger element Euler system. The main obstruction to apply the machinery of~\cite{mr02} directly in our setting is the fact that when $r>1$, what Mazur and Rubin call the canonical Selmer structure in loc.cit. produces a Selmer group too big to control using the theory developed in~\cite{mr02}. In order to deal with this matter we proceed as follows:
 \begin{itemize}
 \item[(1)] We first `refine'  the canonical Selmer structure by introducing more restrictive local conditions at $p$. We achieve this by choosing a \emph{line} $\pmb{\al}$ inside a certain local cohomology group at $p$ (see \S\ref{sub:locp}). This step has to do with the issue discussed in Remark~\ref{rem:comparetorubin}(i) below.
 \item[(2)] We introduce what we call the module of \emph{$\pmb{\al}$-restricted Euler systems} (c.f., Definition~\ref{rmkdef:XrestrictedES}). The point in doing so is that, an Euler system (in the sense of Rubin) a priori gives rise to a Kolyvagin system only for the canonical Selmer structure, and not necessarily for the refined Selmer structure that we defined in Step (1). On the other hand, as we verify here, an $\pmb{\al}$-restricted Euler system does give rise to a Kolyvagin system for the refined Selmer structure.
 \item[(3)] We then obtain $\pmb{\al}$-restricted Euler systems starting from Stickelberger elements (c.f.,  Example~\ref{ex:Lrestricted}), which we apply to deduce Theorem A and B above. For this part, one needs to determine the structure of semi-local cohomology groups and using this information, choose a useful collection of homomorphisms~(see~Proposition~\ref{prop:constructhoms} below). The choice of such a collection also appears in~\cite[Proposition D.1.3]{r00}, where Rubin \emph{explicitly} constructs one --though in a different way from ours; see Remark~\ref{rem:comparetorubin} below for a comparison of our construction here to Rubin's work in the case $k=\QQ$~\cite[\S III.4]{r00}. Rubin's construction is useful only when the base field $k$ is $\QQ$, whereas we abstractly show that a collection of homomorphisms with the necessary properties exists for an arbitrary totally real base field $k$.
  \end{itemize}

 The Galois representation (and the Euler system attached to it) which we treat in this paper needs to be handled in a slightly different manner than the case of Rubin-Stark elements (which was studied in~\cite{kbbstark,kbbiwasawa}), as far as the Euler / Kolyvagin system machinery is concerned. In a forthcoming paper~\cite{kbbrankr}, the set up from the current article and that from~\cite{kbb} are combined together to treat the theory of Kolyvagin systems which descend from Euler systems\footnote{More precisely, \emph{Euler systems of rank $r$} in the terminology of~\cite{pr-es}.} for an \emph{arbitrary} self-dual Galois representation whose core Selmer rank is $r>1$.

Our method to improve the results of~\cite{pr-es,r00,mr02} relies on the choice of a rank-one direct summand (which we call $\pmb{\al}$ above) inside the semi-local cohomology group at $p$. This makes our approach seem less natural. We address this issue in Remark~\ref{rem:question} and show that the module generated by the `leading terms' of the Kolyvagin systems constructed this way does \emph{not} depend on the decomposition of the semi-local cohomology group at $p$ into rank-one direct summands; see Theorem~\ref{thm:independent}.

Besides the standard applications (i.e., Theorem A and Theorem B above) of our construction of what we call an \emph{$\pmb{\al}$-restricted Kolyvagin system} (c.f., Definition~\ref{def:KSboth}, Theorem~\ref{thm:kolsys2} and Remark~\ref{rem:differentLs}) out of Stickelberger elements, we also prove the following statement regarding the local Iwasawa theory of Rubin-Stark elements, which was proved in~\cite{kbbiwasawa} assuming the truth of the main conjecture:
\begin{thmc}
\label{thmc}
Let $\psi:G_k\ra\ZZ_p^\times$ be a totally even character. Suppose that  both $\psi$ and $\chi=\omega\psi^{-1}$ satisfy the hypothesis \textup{(A1)}, and assume  that \textup{(A2)} holds as well as Assumption~\ref{assume:brumer}. Then,
$$\textup{char}\left(\wedge^r_\LL H^1(k_p,T^\prime\otimes\LL)/\LL\cdot c_{k_\infty}^{\psi}\right)=\al_{\psi}.$$
\end{thmc}
 Here, $T^\prime=\ZZ_p(1)\otimes\psi^{-1}$ and $c_{k_\infty}^{\psi}:=\{c_{k_n}^{\psi}\}$ is the $\psi$-part of the collection of (conjectural) Rubin-Stark elements (which we assume to exist) along the cyclotomic tower, see~\cite[\S3]{kbbiwasawa} for a precise definition of these elements.

In fact, we are able to prove considerably more than Theorem C in regard of the Rubin-Stark elements. In Theorem~\ref{thm:maincomparison}(i) below, we obtain a relation between the Stickelberger elements and Rubin-Stark elements (note that the existence of the latter is conjectural), making use of the formalism of $\pmb{\al}$-restricted Euler systems we develop in this paper; as well as the rigidity of the module of $\LL$-adic Kolyvagin systems proved in~\cite{kbb}. This, we believe, is interesting on its own right. We also note that, the ``rigidity phenomenon'' which plays an important role for the link we obtain here (between the Stickelberger elements and Rubin-Stark elements) has recently been utilized by Mazur and Rubin~\cite{mrdarmon} (in a rather similar fashion) to prove an important portion of Darmon's refinement of Gross' conjecture.

We finally remark that, thanks to (an appropriate variant of) Proposition~\ref{prop:semilocalstructure} below, we may by-pass the need of appealing to Krasner's Lemma in~\cite{kbbstark,kbbiwasawa} and hence we may remove the hypothesis that $\chi$ is unramified at all primes $\wp\subset k$ above $p$ on the main results of~\cite{kbbstark,kbbiwasawa}.

 \textbf{Notation:} Besides what we have fixed above, the following notation will be in effect throughout.

  For any field $F$, let $G_F$ denote the Galois group of a fixed separable closure $\overline{F}$ of $F$. For any abelian group $A$, write $$A^{\wedge}:=\Hom(\Hom(A,\QQ_p/\ZZ_p),\QQ_p/\ZZ_p)$$ for its $p$-adic completion. Suppose in addition that $\Delta$ acts on $A$, we then write $A^\chi$ for the $\chi$-isotypic component of $A^{\wedge}$.

 For $k_\infty/k$ as above, let $k_n/k$ be the unique subfield of degree $p^n$. We set $\Gamma_n=\Gal(k_n/k)$ and write $L_n=Lk_n$. For any prime $\qq \subset k$, let $k(\qq)$ denote the $p$-part of the ray class field extension of $k$ modulo $\qq$. For any square free integral ideal $\qq_1\cdots \qq_n=\tau \subset k$, we set $k(\tau)$ as the composite $$k(\tau)=k(\qq_1)\cdots k(\qq_n).$$ Set $\Delta_\tau=\Gal(k(\tau)/k)$. We let $L(\tau)=Lk(\tau)$, $k_n(\tau)=k_nk(\tau)$ and $L_n(\tau)=k_nL(\tau)$.

 \section{Local conditions and Selmer groups}
 Much of this section is a review of the Kolyvagin system machinery and the terminology of~\cite{mr02}, which we will refer to until the end of this paper. The reader who is comfortable with the language of~loc.cit. may safely jump to \S\ref{sec:ES} after a glance at \S\S\ref{subsub:lock}-\ref{subsub:lockinfty}, then at Proposition~\ref{prop:compare selmer}, Corollary~\ref{cor:compare over k with class} and Propositions \ref{prop:compare selmer over k_infty}, \ref{prop:modifiedKSrank1}, Theorem~\ref{main-stark} below. We also note that Remark~\ref{rem:question} (particularly Theorem~\ref{thm:independent}) should be of interest for the general understanding of the Kolyvagin system machinery when the core Selmer core rank is greater than one.

 \subsection{Selmer structures on $T=\ZZ_p(\chi)$}
 Below we use the notation that was set in~\S\ref{introduction}. Recall that $\Gamma:=\Gal(k_\infty/k)$ and $\LL:=\ZZ_p[[\Gamma]]$ is the cyclotomic Iwasawa algebra.

We first recall Mazur and Rubin's definition of a \emph{Selmer structure}, in particular the \emph{canonical Selmer structure} on $T$ and $T\otimes\LL$.

\subsubsection{Local conditions}
\label{local conditions}
 Let $R$ be a complete local noetherian ring, and let $M$ be a $R[[G_k]]$-module which is free of finite rank over $R$. In this paper, we will be interested in the case when $R=\LL$ or its certain quotients, and $M$ is $T\otimes\LL$ or its corresponding quotients by ideals of $\LL$. For example, if we start with the Galois representation $T\otimes\LL$   with coefficients in $\LL$, one gets the representation $T$ upon taking the quotient of $T\otimes\LL$ by the augmentation ideal of $\LL$.

For each place $\lambda$ of $k$, a \emph{local condition} $\FF$ (at $\lambda$)  on $M$ is a choice of an $R$-submodule $H^1_{\FF}(k_\lambda,M)$ of $H^1(k_\lambda,M)$. A local condition $\FF$ at $p$ is a choice of an $R$-submodule $H^1_{\FF}(k_p,M)$ of the semi-local cohomology group $H^1(k_p,M):=\oplus_{\wp|p}H^1(k_\wp,M)$, where the direct sum is over all the primes $\wp$ of $k$ which lie above $p$.

For examples of local conditions, see~\cite[Definition 1.1.6 and 3.2.1]{mr02}.

Suppose that $\FF$ is a local condition (at $\lambda$)  on $M$. If $M^{\prime}$ is a submodule of $M$ (resp., $M^{\prime \prime}$ is a quotient module), then $\FF$ induces local conditions (which we still denote by $\FF$) on $M^{\prime}$ (resp., on $M^{\prime \prime}$), by taking $H^1_{\FF}(k_\lambda,M^{\prime})$ (resp., $H^1_{\FF}(k_\lambda,M^{\prime \prime})$) to be the inverse image (resp., the image) of $H^1_{\FF}(k_\lambda,M)$ under the natural maps induced by $$M^{\prime} \hookrightarrow M, \,\,\, \,\,\,\,\,\,\, M \twoheadrightarrow M^{\prime \prime}.$$

\begin{define}
\label{def:propagation}
\emph{Propagation} of a local condition $\FF$ on $M$ to a submodule $M^{\prime}$ (and a quotient $M^{\prime \prime}$ of $M$) is the local condition $\FF$ on $M^{\prime}$ (and on $M^{\prime \prime}$) obtained following the procedure above.
\end{define}
For example, if $I$ is an ideal of $R$, then a local condition on $M$ induces local conditions on $M/IM$ and $M[I]$, by \emph{propagation}.

\subsubsection{Selmer structures and Selmer groups}
\label{sec:selmer structure}
Notation from \S\ref{local conditions} is in effect throughout this section. We will denote $G_{k_{\lambda}}$ by $\mathcal{D}_{\lambda}$, whenever we wish to identify this group with a closed subgroup of $G_{k}$; namely with a particular decomposition group at $\lambda$. We further define $\mathcal{I}_{\lambda} \subset \mathcal{D}_{\lambda}$ to be the inertia group and $\textup{Fr}_{\lambda} \in \mathcal{D}_{\lambda}/\mathcal{I}_{\lambda}$ to be the arithmetic Frobenius element  at $\lambda$.

\begin{define}
\label{selmer structure}
A \emph{Selmer structure} $\FF$ on $M$ is a collection of the following data:
\begin{itemize}
\item A finite set $\Sigma(\FF)$ of places of $k$, including all infinite places and primes above $p$, and all primes where $M$ is ramified.
\item For every $\lambda \in \Sigma(\FF)$, a local condition (in the sense of \S\ref{local conditions}) on $M$ (which we view now as a $R[[\mathcal{D}_{\lambda}]]$-module), i.e., a choice of $R$-submodule $$H^1_{\FF}(k_{\lambda},M) \subset H^1(k_{\lambda},M).$$
 \end{itemize}
\end{define}
If $\lambda \notin \Sigma(\FF)$, we will also write $H^1_{\FF}(k_{\lambda},M)=H^1_{\textup{f}}(k_{\lambda},M)$, where the module $H^1_{\textup{f}}(k_{\lambda},M)$ is the \emph{finite} part of $H^1(k_{\lambda},M)$, defined as in~\cite[Definition 1.1.6]{mr02}.

\begin{define}
\label{cartier dual}
Define the \emph{Cartier dual} of $M$ to be the $R[[G_k]]$-module $$M^*:=\textup{Hom}(M,\mu_{p^{\infty}}),$$ where $\mu_{p^{\infty}}$ stands for the $p$-power roots of unity.
\end{define}
Let $\lambda$ be a prime of $k$. There is a perfect pairing $$<\,,\,>_\lambda\,:H^1(k_\lambda,M) \times H^1(k_\lambda,M^*) \lra H^2(k_\lambda,\mu_{p^{\infty}}) \stackrel{\sim}{\lra}\QQ_p/\ZZ_p,$$
called the local Tate pairing.
\begin{define}
\label{dual local condition}
The \emph{dual local condition} $\FF^*$ on $M^*$ of a local  condition $\FF$ on $M$ is defined so that $H^1_{\FF^*}(k_\lambda,M^*)$ is the orthogonal complement of $H^1_{\FF}(k_\lambda,M)$ with respect to the local Tate pairing $<\,,\,>_\lambda$.

 The \emph{dual Selmer structure} $\FF^*$ is defined by setting $\Sigma(\FF^*)=\Sigma(\FF)$ and choosing the local conditions on $M^*$ as the dual local conditions $H^1_{\FF^*}(k_\lambda,M^*)=H^1_{\FF}(k_\lambda,M)^{\perp}$ at every prime $\lambda \in \Sigma(\FF^*)$.
\end{define}

\begin{define}
\label{selmer group}
If $\FF$ is a Selmer structure on $M$, we define the \emph{Selmer module} $H^1_{\FF}(k,M)$ to be the kernel of the sum of the restriction maps:
\be\label{eqn:selmergroup}H^1_{\FF}(k,M):=\ker\left(H^1(\textup{Gal}(k_{\Sigma(\FF)}/k),M) \ra \bigoplus_{\lambda \in \Sigma(\FF)}\frac{H^1(k_{\lambda},M)}{H^1_{\FF}(k_{\lambda},M)}\right).\ee
Here, $k_{\Sigma(\FF)}$ is the maximal extension of $k$ which is unramified outside $\Sigma(\FF)$. We also define the dual Selmer module $H^1_{\FF^*}(k,M^*)$ in a similar fashion; just replace $M$ by $M^*$ and $\FF$ by $\FF^*$ in (\ref{eqn:selmergroup}).
\end{define}

\begin{example}
\label{example:canonical selmer}
In this example we recall~\cite[Definition 3.2.1 and  5.3.2]{mr02}.
\begin{itemize}
\item[(i)] Let $R=\ZZ_p$ and let ${M}$ be a free $R$-module endowed with a continuous action of $G_k$, which is unramified outside a finite set of places of $k$.  We define a Selmer structure $\FFc$ on ${M}$ by setting
\begin{itemize}
\item $\Sigma(\FFc)=\{\lambda: M \hbox{ is ramified at } \lambda\}\cup\{\wp|p\}\cup\{v|\infty\}$,
\item $H^1_{\FFc}(k_\lambda, {M})=\left\{\begin{array}{lcl} H^1_{\textup{f}}(k_\lambda,M)&,& \hbox{ if } \lambda \in \Sigma(\FFc), \lambda\nmid p\infty,\\
H^1(k_\lambda,M)&,& \hbox{ if }  \lambda|p.
\end{array} \right.$
\end{itemize}
Here, $H^1_{\finite}(k_\lambda, M):=\ker\{H^1(k_\lambda, M) \ra H^1(k_\lambda, M\otimes\QQ_p)\}$for every $\lambda\nmid p\ff_\chi$.

The Selmer structure $\FFc$ is called the \emph{canonical Selmer structure} on ${M}$.
\item[(ii)] Let now $R=\LL$ be the cyclotomic Iwasawa algebra, and let $\mathbb{M}$ be a free $R$-module endowed with a continuos action of $G_k$, which is unramified outside a finite set of places of $k$.  We define a Selmer structure $\FF_\LL$ on $\mathbb{M}$ by setting
\begin{itemize}
\item $\Sigma(\FF_\LL)=\{\lambda: \mathbb{M} \hbox{ is ramified at } \lambda\}\cup\{\wp|p\}\cup\{v|\infty\}$,
\item $H^1_{\FF_\LL}(k_\lambda, \mathbb {M})=H^1(k_\lambda, \mathbb{M})$ for $\lambda \in \Sigma(\FF_\LL)$
\end{itemize}
 The Selmer structure $\FF_\LL$ is called the \emph{canonical Selmer structure} on $\mathbb{M}$.

As in Definition~\ref{def:propagation}, the induced Selmer structure on the quotients $\mathbb{M}/I\mathbb{M}$ is still denoted by $\FF_\LL$.  Note that $H^1_{\FF_\LL}(k_\lambda, \mathbb{M}/I\mathbb{M})$ will not usually be the same as $H^1(k_\lambda, \mathbb{M}/I\mathbb{M})$. In particular, when $I$ is the augmentation ideal inside $\Lambda$, $\FF_\LL$ on $\mathbb{M}$ will not always propagate to $\FFc$ on $M=\mathbb{M}\otimes\LL/I.$

However, when $M=T$ and $\mathbb{M}=T\otimes\LL$ as in \S\ref{introduction}, it is not hard to see that $\FF_\LL$ \emph{does} propagate to $\FFc$.
\end{itemize}
\end{example}
\begin{rem}
\label{rem:canonical selmer}
When $R=\LL$ and $\ \mathbb{M}=T\otimes\LL$ with $T=\ZZ_p(\chi)$, we will see in~\S\ref{subsubsec:kolsyslambda1} that the Selmer structure $\FFc$ of~\cite[\S2.1]{kbb}  on the quotients $T\otimes\LL/(f)$ may be identified, under the hypothesis (A1) on $\chi$, with the propagation of $\FF_\LL$  to the quotients $T\otimes\LL/(f)$, for every distinguished polynomial $f$ inside $\LL$.
\end{rem}

\begin{define}
\label{def:selmer triple}
A \emph{Selmer triple} is a triple $(M,\FF,\PP)$ where $\FF$ is a Selmer structure on $M$ and $\PP$ is a set of rational primes, disjoint from $\Sigma(\FF)$.
\end{define}

\begin{rem}
In our setting, i.e., when the Galois representation in question is $T\otimes\LL$ or its quotients by ideals of $\LL$, one may explicitly compute the cohomology groups in terms of certain groups of homomorphisms (c.f., \cite[\S I.6.1-3]{r00}). Nevertheless, we will insist on using the cohomological language for the sake of notational consistency with~\cite{mr02} from which we borrow the main technical results. We also hope that the similarity of the ideas applied here and in \cite{kbbstark,kbbiwasawa} are more apparent this way.
\end{rem}

\subsection{Computing Selmer groups explicitly}
In this section, we give an explicit description of the Selmer groups for the $G_k$-representations $T=\ZZ_p(\chi)$ (resp., for $T\otimes\LL$) and  $T^*=\pmb{\mu}_{p^\infty}\otimes\chi^{-1}$ (resp., for $(T\otimes\LL)^*$); following~\cite[\S I.6.2]{r00} and \cite[\S6.1]{mr02}.
\subsubsection{Selmer groups over $k$}
Recall that $L$ is the CM field cut by $\chi$. For any $m \in \ZZ^+$, it follows (as in~\cite[\S6.1]{mr02}) from the inflation-restriction sequence that
\be
\label{eqn:explicit1}
H^1(k,T/p^mT)= H^1(k,\ZZ/p^m\ZZ(\chi))\cong \Hom(G_L,\ZZ/p^m\ZZ)^{\chi^{-1}},
\ee
and similarly for every prime $\lambda$ of $k$,
\be
\label{eqn:explicitlocal1}
H^1(k_\lambda,T/p^mT) \cong \left(\bigoplus_{\qq|\lambda}\Hom(G_{L_{\qq}},\ZZ/p^m\ZZ)\right)^{\chi^{-1}}.
\ee
 Therefore, for the semi-local cohomology at a rational prime $\ell$, we have
\be
\label{eqn:explicitsemilocal1}
H^1(k_\ell,T/p^mT)\cong \bigoplus_{\lambda|\ell}\left(\bigoplus_{\qq|\lambda}\Hom(G_{L_{\qq}},\ZZ/p^m\ZZ)\right)^{\chi^{-1}}.
\ee
Passing to the inverse limit, we obtain
\be
\label{eqn:expT}
H^1(k,T)\cong \Hom(G_L,\ZZ_p)^{\chi^{-1}},
\ee
and
\be
\label{eqn:expT2}
H^1(k_\ell,T)\cong \bigoplus_{\lambda|\ell}\left(\bigoplus_{\qq|\lambda}\Hom(G_{L_{\qq}},\ZZ_p)\right)^{\chi^{-1}}.
\ee

For the dual representation $T^*$, we have by the inflation-restriction sequence and by Kummer theory
\be
\label{eqn:explicit2}
H^1(k,T^*[p^m])= H^1(k,\pmb{\mu}_{p^m}\otimes\chi^{-1})\cong \left(L^{\times}/(L^{\times})^{p^m}\right)^{\chi},
\ee
and similarly for every prime $\lambda\subset k$,
\be
\label{eqn:explicitlocal2}
H^1(k_\lambda,T^*[p^m])\cong \left(L_{\lambda}^{\times}/(L_{\lambda}^{\times})^{p^m}\right)^{\chi}.
\ee
Also for the semi-local cohomology, we have
\be
\label{eqn:explicitsemilocal2}
H^1(k_\ell,T^*[p^m])\cong \left(L_{\ell}^{\times}/(L_{\ell}^{\times})^{p^m}\right)^{\chi},
\ee
where $L_\lambda:=L\otimes_k k_\lambda$, the sum of the completions of $L$ at the primes above $\lambda$, and $L_\ell:=L\otimes_{\QQ}\QQ_\ell$. Taking direct limits, we see that
\be
\label{eqn:expT*}
H^1(k,T^*)\cong \left(L^{\times}\otimes\QQ_p/\ZZ_p\right)^{\chi},
\ee
and
\be
\label{eqn:expT*2}
H^1(k_\ell,T^*)\cong \left(L_{\ell}^{\times}\otimes\QQ_p/\ZZ_p\right)^{\chi}.
\ee

\begin{prop}
\label{prop:local conditions explicit}
The canonical Selmer structure $\FFc$ on $T$ (resp., $\FFc^*$ on $T^*$) is given by
\begin{itemize}
\item  $\Sigma(\FFc)=\Sigma(\FFc^*)=\{\lambda: \lambda | p\ff_\chi\} \cup \{v| \infty\}$,
\end{itemize}
and by setting \textup{(}using the identifications above\textup{)}:
\begin{itemize}
\item $H^1_{\FFc}(k_\ell,T)=\left(\bigoplus_{\qq|\ell}\Hom(G_{L_{\qq}}/\mathcal{I}_{\qq},\ZZ_p)\right)^{\chi^{-1}}$, \\$H^1_{\FFc^*}(k_\ell,T^*)=(\mathcal{O}_{L,\ell}^\times\otimes\QQ_p/\ZZ_p)^{\chi}$, if $\ell\neq p$,
\item $H^1_{\FFc}(k_p,T)=H^1(k_p,T)$, \\ $H^1_{\FFc^*}(k_p,T^*)=0$.
\end{itemize}
\end{prop}
Here, $\mathcal{I}_{\qq}$ stands for a fixed inertia group at $\qq$, and $\mathcal{O}_{L,\ell}:=\mathcal{O}_{L}\otimes\ZZ_\ell$ is the sum of the local units inside $L_\ell=\oplus_{\qq|\ell} L_{\qq}$.
\begin{proof}
This is proved in~\cite[\S I.6.2-3]{r00}.
\end{proof}

\begin{define}
\label{def:classical Selmer k}
We define the classical Selmer structure $\FF_{\cl}$ on $T$ (and $\FF_{\cl}^*$ on $T^*$) by setting $\Sigma(\FF_{\cl})=\Sigma(\FFc)$,
and by letting
\begin{itemize}
\item $H^1_{\FF_{\cl}}(k_\ell,T)=H^1_{\FFc}(k_\ell,T)$, and \\
$H^1_{\FF_{\cl}^*}(k_\ell,T^*)=H^1_{\FFc^*}(k_\ell,T^*)$, if $\ell\neq p$,
\item $H^1_{\FF_{\cl}}(k_p,T)=\left(\bigoplus_{\qq|p}\Hom(G_{L_{\qq}}/\mathcal{I}_{\qq},\ZZ_p)\right)^{\chi^{-1}}$, and\\
$H^1_{\FF_{\cl}^*}(k_p,T^*)=(\mathcal{O}_{L,p}^\times\otimes\QQ_p/\ZZ_p)^{\chi}$.
\end{itemize}
\end{define}
\begin{rem}
\label{rem:compare classical with can}
If we assume that (A1) holds, it follows from the proof of~\cite[Proposition III.2.6]{r00} (see also~\cite[Lemma 6.1.2]{mr02}) that $H^1_{\FF_{\cl}}(k_p,T)=0$ and $H^1_{\FF_{\cl}^*}(k_p,T^*)=H^1(k_p,T^*)$. We therefore have the following exact sequences
$$\xymatrix@R=.3cm{0\ar[r]&H^1_{\FF_{\cl}}(k,T)\ar[r]& H^1_{\FFc}(k,T)\ar[r]^{\textup{loc}_p}&H^1(k_p,T)\\
0\ar[r]&H^1_{\FFc^*}(k,T^*)\ar[r]& H^1_{\FF_{\cl}^*}(k,T^*)\ar[r]^{\textup{loc}_p^*}&H^1(k_p,T^*)
}$$
Furthermore, the image of $\textup{loc}_p$ is the orthogonal complement of the image of $\textup{loc}_p^*$, by the Poitou-Tate global duality theorem. We note that the classical Selmer group $H^1_{\FF_{\cl}}(k,T)$ (resp., $H^1_{\FF_{\cl}^*}(k,T^*)$) is denoted by $\mathcal{S}(k,T)$ (resp., by $\mathcal{S}(k,W^*)$) in~\cite{r00}.
\end{rem}

\begin{prop}
\label{prop:classical selmer explicit}
Let $A_L$ denote the $p$-part of the ideal class group of $L$. Then, $H^1_{\FF_{\cl}}(k,T)=0$ and $H^1_{\FF_{\cl}^*}(k,T^*)\cong A_L^\chi$.
\end{prop}

\begin{proof}
Proposition 6.1.3 of~\cite{mr02} gives $$H^1_{\FF_{\cl}}(k,T)=\varprojlim_m \Hom(A_L^\chi,\ZZ/p^m\ZZ)=\Hom(A_L^\chi,\ZZ_p);$$we note that the propagation of $\FF_{\cl}$ to $\ZZ/p^m\ZZ(\chi)$ coincides with the Selmer structure $\FF^*$ of \cite[\S6.1]{mr02}. Since $A_L^\chi$ is finite, it follows that $H^1_{\FF_{\cl}}(k,T)=0$.

Similarly, the propagation of $\FF_{\cl}^*$ to $\pmb{\mu}_{p^m}\otimes\chi^{-1}$ coincides with the Selmer structure $\FF$ of  \cite[\S 6.1]{mr02}. It therefore follows from~\cite[Proposition 6.1.3]{mr02} that there is an exact sequence
$$0\lra \left(\mathcal{O}_L^{\times}/(\mathcal{O}_L^{\times})^{p^m}\right)^{\chi}\lra  H^1_{\FF_{\cl}^*}(k,T^*[p^m]) \lra A_L[p^m]^\chi\lra 0.$$
Taking the direct limit with respect to $m$, we obtain the following exact sequence:
$$0\lra \left(\mathcal{O}_L^{\times}\otimes\QQ_p/\ZZ_p\right)^{\chi}\lra  H^1_{\FF_{\cl}^*}(k,T^*) \lra A_L^\chi\lra 0.$$
 Since $\chi$ is totally odd, it follows from~\cite[Proposition I.3.4]{tate} that $(\mathcal{O}_L^{\times})^{\chi}$ is finite, hence $\left(\mathcal{O}_L^{\times}\otimes\QQ_p/\ZZ_p\right)^{\chi}=0$. This completes the proof of the Proposition.
\end{proof}

\subsubsection{Selmer groups over $k_\infty$}
Let $k_n$ denote the unique subfield $k_\infty$, which is of degree $p^n$ over $k$. We also set $L_n=L\cdot k_n$. Repeating the arguments of the previous section (replacing the totally real field $k$ with the totally real field $k_n$), we prove the following:

\begin{lemma}
\label{lem:iwasawa-selmer-explicit}
There is a canonical identification $$\varinjlim_{n} H^1_{\FF_{\cl}}(k_n,T^*)=\varinjlim_{n}A_{L_n}^\chi.$$
\end{lemma}

 \subsection{Modifying the local conditions at $p$}
\label{sub:locp}
When the core Selmer rank of a Selmer structure (in the sense of~\cite{mr02}, see also \S\ref{subsec:KS1} below) is greater than one, it produces a Selmer group which is difficult to control using the Kolyvagin system machinery of~\cite{mr02}. As we will see in~\S\ref{subsec:KS1}, the Selmer structure $\FFc$ on $T$ (resp., $\FF_{\LL}$ on $T\otimes\LL$) has core Selmer rank $r=[k:\QQ]$. Hence, to be able to utilize the Kolyvagin system machinery, we will need to modify $\FFc$ and $\FF_\LL$ appropriately. This is what we do in this section.

Throughout this section we assume (A1) and (A2).
\subsubsection{Local conditions at $p$ over $k$}
\label{subsub:lock}
 \begin{lemma}
 \label{lemma:free for k}
Under our running hypotheses, $$H^1(k_p,T):=\bigoplus_{\wp|p}H^1(k_\wp,T)$$ is a free $\ZZ_p$-module of rank $r=[k:\QQ]$.
 \end{lemma}

 \begin{proof}
We first prove this using the general structure theory of semi-local cohomology groups at $p$. All the references in this proof are to~\cite[Appendix A]{kbbiwasawa}. We note that the results quoted in loc.cit. are originally due to Benois, Colmez, Herr and Perrin-Riou.

By Theorem A.8(i),  the $\LL$-torsion submodule $H^1(k_p,T\otimes\LL)_{\textup{tors}}$ is isomorphic to $\oplus_{\wp|p}T^{H_{k_\wp}}$, where $H_{k_\wp}=\Gal(\overline{k_\wp}/k_{\wp,\infty})$. Since we assume (A1), it follows that $H^1(k_p,T\otimes\LL)_{\textup{tors}}=0$. Theorem A.8(ii) now concludes that  the $\LL$-module  $H^1(k_p,T\otimes\LL)$ is free rank $r$. Furthermore,
$$\textup{coker}[H^1(k_p,T\otimes\LL)\lra H^1(k_p,T)]=H^2(k_p,T\otimes\LL)[\gamma-1],$$
 where $\gamma$ is any topological generator of $\Gamma$. However, it follows from~\cite[Lemma 2.11]{kbb} that $H^2(k_p,T\otimes\LL)=0$, hence the map $$H^1(k_p,T\otimes\LL)\lra H^1(k_p,T)$$ is surjective. Lemma now follows.
 \end{proof}
\begin{rem}
\label{rem:alternative proof}
There is a more direct proof of Lemma~\ref{lemma:free for k}. In this remark, we include this alternative proof of this lemma.

By the explicit description of the semi-local cohomology groups in (\ref{eqn:expT2})
$$H^1(k_p,T)\cong \bigoplus_{\wp|p}\left(\bigoplus_{\qq|\wp}\Hom(G_{L_{\qq}},\ZZ_p)\right)^{\chi^{-1}}.$$
It follows at once from this description that $H^1(k_p,T)$ is $\ZZ_p$-torsion free, hence free. Further, since $\ZZ_p$ is an abelian group we may rewrite the equality above as
$$H^1(k_p,T)\cong \bigoplus_{\wp|p}\left(\bigoplus_{\qq|\wp}\Hom(G_{L_{\qq}}^{\textup{ab}},\ZZ_p)\right)^{\chi^{-1}},$$
where $G_{L_{\qq}}^{\textup{ab}}$ stands for the abelianization of $G_{L_{\qq}}$. By local class field theory $G_{L_{\qq}}^{\textup{ab}} \cong L_{\qq}^{\wedge}$, the $p$-adic completion of the multiplicative group of $L_{\qq}$. Further, the valuation map $\textup{val}_{\qq}$ gives an isomorphism
$$L_{\qq}^\times \stackrel{\textup{val}_{\qq}}{\lra}\ZZ_p\oplus\mathcal{O}_{L_{\qq}}^{\times,\wedge}.$$
We therefore have
\begin{align*}H^1(k_p,T)& \cong\Hom\left(\bigoplus_{\qq|p}(\ZZ_p\oplus\mathcal{O}_{L_{\qq}}^{\times,\wedge}),\ZZ_p\right)^{\chi^{-1}}\\
& \cong \Hom\left(\left(\oplus_{\qq|p}\ZZ_p\right)^\chi\oplus\left(\oplus_{\qq|p}\mathcal{O}_{L_{\qq}}^{\times,\wedge}\right)^{\chi},\ZZ_p\right).\end{align*}
It follows from (A1) that $\left(\oplus_{\qq|p}\ZZ_p\right)^\chi=0$, hence $$H^1(k_p,T)\cong \Hom\left(\left(\oplus_{\qq|p}\mathcal{O}_{L_{\qq}}^{\times,\wedge}\right)^{\chi},\ZZ_p\right).$$ To prove the Lemma, it suffices to check that the $\QQ_p$-dimension of $V:=\left(\oplus_{\qq|p}\mathcal{O}_{L_{\qq}}^{\times,\wedge}\otimes\QQ_p\right)^{\chi}$ is equal to $r$. The $p$-adic  logarithm gives a homomorphism $\mathcal{O}_{L_{\qq}}^{\times,\wedge} \ra \mathcal{O}_{L_{\qq}}$ with finite kernel and cokernel. Hence
$$V=\left(\oplus_{\qq|p}\mathcal{O}_{L_{\qq}}^{\times,\wedge}\otimes\QQ_p\right)^{\chi}=\left(\oplus_{\qq|p}\mathcal{O}_{L_{\qq}}\otimes\QQ_p\right)^{\chi}=(L\otimes\QQ_p)^{\chi}$$ and therefore the $\QQ_p$-dimension of $V$ equals $r$ by the normal basis theorem.
\end{rem}
 \begin{define}
 \label{def:line}
 Fix a $\ZZ_p$-direct summand $\mathcal{L} \subset H^1(k_p,T)$ such that $\al$ is free of rank one. Fix also a generator $\varphi=\varphi_\mathcal{L}$ of $\mathcal{L}$. Define the \emph{$\al$-modified Selmer structure} $\FF_\al$ on $T$ as follows:

 \begin{itemize}
 \item $\Sigma(\FF_\al)=\Sigma(\FFc)$,
 \item if $\lambda \nmid p$, $H^1_{\FF_\al}(k_\lambda, T)=H^1_{\FFc}(k_\lambda,T)$,
 \item $H^1_{\FF_\al}(k_p,T):=\al \subset H^1(k_p,T)=H^1_{\FFc}(k_p,T)$.
 \end{itemize}
 \end{define}

 \subsubsection{Local conditions at $p$ over $k_\infty$}
 \label{subsub:lockinfty}

 Set $\Gamma=\Gal(k_{\infty}/k)$ as before. Let $k_{\wp}$ denote the completion of $k$ at $\wp$, and let  $k_{\wp,\infty}$ denote the cyclotomic $\ZZ_p$-extension of $k_{\wp}$. Since we assume (A2), we may identify $\Gal(k_{\wp,\infty}/k_{\wp})$ by $\Gamma$ for each $\wp|p$ and henceforth $\Gamma$ will stand for any of these Galois groups. Let $\LL=\ZZ_p[[\Gamma]]$ be the cyclotomic Iwasawa algebra, as usual. We also fix a topological generator $\gamma$ of $\Gamma$, and we set $\xx=\gamma-1$. We will occasionally identify $\LL$ by the power series ring $\ZZ_p[[\xx]]$.
 \begin{lemma}
 \label{lemma:free for k_infty}
 Under the assumptions \textup{(A1)} and \textup{(A2)},  $$H^1(k_p,T\otimes\LL):=\oplus_{\wp|p}H^1(k_\wp,T\otimes\LL)$$ is a free $\LL$-module of rank $r$.
 \end{lemma}

 \begin{proof}
This is checked in the first part of the proof of Lemma~\ref{lemma:free for k}.
 \end{proof}

 \begin{define}
 \label{def:line over k_infty}
 Fix a $\LL$-rank one direct summand $\mathbb{L} \subset H^1(k_p,T\otimes\LL)$ such that $\all$ maps onto $\al$ under the projection
 \be\label{eqn:projection}\xymatrix{H^1(k_p,T\otimes\LL)\ar@{->>}[r] &H^1(k_p,T).}\ee
  Fix also a generator $\Phi=\Phi_\mathbb{L}$ of $\mathbb{L}$, which maps to $\varphi=\varphi_\al$ under the projection (\ref{eqn:projection}). Define the \emph{$\all$-modified Selmer structure} $\FF_\all$ on $T\otimes\LL$ as follows:

 \begin{itemize}
 \item $\Sigma(\FF_\all)=\Sigma(\FF_\LL)$,
 \item if $\lambda \nmid p$, $H^1_{\FF_\all}(k_\lambda, T\otimes\LL)=H^1_{\FF_\LL}(k_\lambda,T\otimes\LL)$,
 \item $H^1_{\FF_\all}(k_p,T\otimes\LL):=\all \subset H^1(k_p,T\otimes\LL)=H^1_{\FF_\LL}(k_p,T\otimes\LL)$.
 \end{itemize}
 \end{define}

\begin{rem}
\label{rmk:FFall propagates to FFal}
By definition, the image of $H^1_{\FF_{\all}}(k_p,T\otimes\LL)$ is $H^1_{\FF_{\al}}(k_p,T)$ under the map $H^1(k_p,T\otimes\LL)\ra H^1(k_p,T)$. Further, it follows from \cite[Lemma 5.3.1(ii)]{mr02} for $\ell \neq p$ that $H^1_{\FF_{\all}}(k_\ell,T\otimes\LL)$ also maps to $H^1_{\FF_{\al}}(k_\ell,T)$ under the natural map $H^1(k_\ell,T\otimes\LL) \ra H^1(k_\ell,T)$. In other words, $\FF_{\all}$ propagates to $\FF_{\al}$, and there is an induced map $$H^1_{\FF_{\all}}(k,T\otimes\LL) \lra H^1_{\FF_{\al}}(k,T).$$
\end{rem}
\subsection{Global duality and a comparison of Selmer groups}
In this section, we compare the classical Selmer group (which we wish to relate to $L$-values) to the modified Selmer groups (for which we will apply the Kolyvagin system machinery and which we will compute in terms of $L$-values). The necessary tool to accomplish this comparison is Poitou-Tate global duality.
\subsubsection{Comparison over $k$}
\label{subsubsec:compareselmer1}
The definition of the modified Selmer structure $\FF_{\al}$ and Remark~\ref{rem:compare classical with can} gives us the following exact sequences:
$$\xymatrix @R=.3cm{0\ar[r]&H^1_{\FF_{\cl}}(k,T)\ar[r]& H^1_{\FF_{\al}}(k,T)\ar[r]^(.6){ \textup{loc}_p}&\al\\
0\ar[r]&H^1_{\FF_{\al}^*}(k,T^*)\ar[r]&
H^1_{\FF_{\cl}^*}(k,T^*)\ar[r]^(.5){\textup{loc}_p^*}&\frac{H^1_{\FF_{\cl}^*}(k_p,T^*)}{H^1_{\FF_{\al}^*}(k_p,T^*)}}$$
Poitou-Tate global duality (c.f., \cite[Theorem I.7.3]{r00},
\cite[Theorem I.4.10]{milne}) states that the image of
$\textup{loc}_p$ is the orthogonal compliment of the image of
$\textup{loc}_p^*$ with respect to the local Tate pairing. Using
this fact, together with Proposition~\ref{prop:classical selmer
explicit}, one may prove the following Proposition for
$T=\ZZ_p(\chi)$ as above. Note that $H^1_{\FF_{\textup{cl}}}(k,T)=0$
by Proposition~\ref{prop:classical selmer explicit}.
See~\cite[Theorem I.7.3]{r00} for further details:

\begin{prop}
\label{prop:compare selmer}
We have an exact sequence
$$0\ra  H^1_{\FF_{\al}}(k,T)\stackrel{\textup{loc}_p}{\lra}\al\stackrel{(\textup{loc}_p^*)^{\vee}}{\lra}  \left( H^1_{\FF_{\cl}^*}(k,T^*)\right)^{\vee}\lra\left(H^1_{\FF_{\al}^*}(k,T^*)\right)^{\vee}\ra 0,$$ where the map $(\textup{loc}_p^*)^{\vee}$ is induced from localization at $p$ and the local Tate pairing between $H^1(k_p,T)$ and $H^1(k_p,T^*)$.
\end{prop}

Suppose $c \in  H^1_{\FF_{\al}}(k,T)$ is any class. We still write $c$ for the image of the class $c$ inside $\al=H^1_{\FF_{\al}}(k_p,T)$ under the (injective) map $\textup{loc}_p$.
\begin{cor}
\label{cor:compare over k with class}
The following sequence is exact:
$$ 
0\ra \frac{H^1_{\FF_{\al}}(k,T)}{\ZZ_p\cdot c}\stackrel{\textup{loc}_p}{\ra}\frac{\al}{\ZZ_p\cdot c} \stackrel{(\textup{loc}_p^*)^{\vee}}{\lra}  \left( H^1_{\FF_{\cl}^*}(k,T^*)\right)^{\vee}\ra\left(H^1_{\FF_{\al}^*}(k,T^*)\right)^{\vee}\ra 0.$$
\end{cor}
\subsubsection{Comparison over $k_\infty$}
\label{subsubsec:compareselmer2}
Repeating the argument of Proposition~\ref{prop:compare selmer} for each field $k_n$ (instead of $k$) and passing to inverse limit we obtain the following:

\begin{prop}
\label{prop:compare selmer over k_infty}
Both of the following sequences of $\LL$-modules are exact:
\begin{itemize}
\item[(i)]
$0\ra{\scriptstyle H^1_{\FF_{\all}}(k,T\otimes\LL)}\stackrel{\textup{loc}_p}{\lra}\all\lra {\scriptstyle \left( \varinjlim_n A_{L_n}^\chi\right)^\vee}\lra{\scriptstyle \left(H^1_{\FF_{\all}^*}(k,(T\otimes\LL)^*)\right)^{\vee}}\ra 0, $

\item[(ii)] For any class $c \in H^1(k,T\otimes\LL)$,
$$0 \ra \frac{{\scriptstyle  H^1_{\FF_{\all}}(k,T\otimes\LL)}}{\LL\cdot c}\stackrel{\textup{loc}_p}{\lra}{\frac{\all}{\LL\cdot c}} {\lra} {\scriptstyle \left(\varinjlim_n A_{L_n}^{\chi}\right)^{\vee}}\lra{\scriptstyle\left(H^1_{\FF_{\all}^*}(k,(T\otimes\LL)^*)\right)^{\vee}}\lra 0.$$

\end{itemize}
\end{prop}
\begin{proof}
We only  give a sketch since similar versions of this Proposition are already available in the literature (c.f., \cite[Theorem I.7.3 and III.2.10]{r00}, \cite[\S III.1.7]{deshalit}).

Thanks to the argument of Proposition~\ref{prop:compare selmer} and \cite[Proposition B.1.1]{r00}, there is an exact sequence
\begin{align*}
0 \lra {\varprojlim_n H^1_{\FF_{\al_n}}(k_n,T)}{\lra}{\varprojlim_n\al_n}{\lra} &{\left( \varinjlim_n H^1_{\FF_{\cl}^*}(k_n,T^*)\right)^{\vee}}\\
&\lra{\left(\varinjlim_n H^1_{\FF_{\al_n}^*}(k_n,T^*)\right)^{\vee}}\lra 0,
\end{align*}
where $\al_n$ is the image of $\all$ under the natural map $$H^1(k_p,T\otimes\LL)\lra H^1((k_n)_p,T).$$
By definition, $\varprojlim_n\al_n=\all$, and by~\cite[Lemma 5.3.1]{mr02} (or rather by its proof) it follows that there is a canonical isomorphism
$$\varprojlim_n H^1_{\FF_{\al_n}}(k_n,T) \cong H^1_{\FF_{\all}}(k,T\otimes\LL).$$
Furthermore, by Lemma~\ref{lem:iwasawa-selmer-explicit}, $ \varinjlim_n H^1_{\FF_{\cl}^*}(k_n,T^*)=\varinjlim_n A_{L_n}^\chi$. Finally, by Shapiro's lemma $$H^1(k_n,T^*)=H^1(k,T^*\otimes\ZZ_p[\Gamma_n]),$$ where $\Gamma_n=\Gal(k_n/k)$, hence \be\label{nonsense}\varinjlim_n H^1(k_n,T^*)=H^1(k, \varinjlim_n T^*\otimes\ZZ_p[\Gamma_n]).\ee Now, using the fact that the functors $-\otimes_{\ZZ_p}\ZZ_p[\Gamma_n]$ and $\Hom_{\ZZ_p}(\ZZ_p[\Gamma_n],-)$ are adjoint functors (we  drop the subscripts below and write $\otimes$ and Hom for short), it follows that \begin{align*}(T\otimes\LL)^*&:=\Hom(\varprojlim_nT\otimes\ZZ_p[\Gamma_n],\QQ_p/\ZZ_p)(1)\\&\cong\varinjlim_n\Hom(T,\Hom(\ZZ_p[\Gamma_n],\QQ_p/\ZZ_p))(1)\\&\cong\varinjlim_n\Hom(T,\QQ_p/\ZZ_p[\Gamma_n])(1)\\&\cong \varinjlim_n\Hom(T,\QQ_p/\ZZ_p)(1)\otimes\ZZ_p[\Gamma_n]=:\varinjlim_n T^*\otimes\ZZ_p[\Gamma_n],
\end{align*}
where the isomorphism of the modules in the second and the third line comes from  the isomorphism
$$\xymatrix  @C=10pt @R=.1pt{\Hom(\ZZ_p[\Gamma_n],\QQ_p/\ZZ_p) \ar[r]^(.6){\sim}& \QQ_p/\ZZ_p[\Gamma_n]\\
f\ar@{{|}->}[r]& \sum_{\gamma \in \Gamma_n} f(\gamma)\cdot\gamma
}$$
of $\ZZ_p[\Gamma_n]$-modules. This and (\ref{nonsense}) (together with its semi-local analogue) show at once that
$$\varinjlim_n H^1_{\FF_{\al_n}^*}(k_n,T^*)=H^1_{\FF_{\all}^*}(k,(T\otimes\LL)^*).$$
This completes the proof of (i), and (ii) follows trivially from (i).
\end{proof}

\subsection{Kolyvagin systems for modified Selmer groups}
\label{subsec:KS1}
This section closely follows the exposition of~\cite[\S1.2]{kbbstark} and \cite[\S2.5]{kbbiwasawa}.
We assume (A1) and (A2) throughout this section.
\begin{rem}
\label{hypo holds}
It is straightforward to verify that the following hypotheses (which were introduced in~\cite[\S 3.5]{mr02}) hold for $T=\ZZ_p(\chi)$:
\begin{itemize}
\item[($\mathbf{H.1}$)] The residual $\mathbb{F}_p [[G_k]]$-representation $T/pT$ is absolutely irreducible.
\item[($\mathbf{H.2}$)] There is a $\tau \in G_k$ such that $\tau=1 \hbox{ on } \mu _{p ^{\infty}}$ and $T/(\tau -1)T$ is free of rank one over $\ZZ _p$.
\item[($\mathbf{H.3}$)] $H^{0}(k , T/pT)$ = $H^{0}(k , T^{*}[p])$ =  0.
\item[($\mathbf{H.4}$)] $\textup{Hom}_{\mathbb{F} _{p}[[G_k]]}(T/pT, T ^{*}[p])=0$.
\end{itemize}
We remark that the hypothesis $(\mathbf{H.3})$ above is implied by what Mazur and Rubin call $({H.3})$ (c.f.,~\cite[Lemma 3.5.2]{mr02}). The hypothesis (\hthree) above is sufficient for our purposes.
\end{rem}

Let $\PP$ denote the set whose elements are prime ideals of $k$ which are prime to $p\ff_{\chi}$.  For each positive integer $m$ and $n$, let
$$\PP_{m+n}=\{\qq \in \PP: \qq \hbox{ splits completely in } L(\mu_{p^{m+n+1}})/k\}$$
be a subset of $\PP$. Note that $\PP_{m+n}$ is exactly the set of primes determined by~\cite[Definition IV.1.1]{r00} when $T=\ZZ_p(\chi)$. The hypothesis (H.5) of~\cite[\S3.5]{mr02} holds with this choice of $\PP$. Let $\NN=\NN(\PP)$ (resp., $\NN_j=\NN(\PP_j)\subset \NN$) be the square free products of primes $\qq \in \PP$ (resp., in $\PP_j$), with the convention that $1 \in \NN_j\subset \NN$.

Using \cite[Lemma 3.7.1]{mr02}, one may also check that $\FFc$ and $\FF_{\al}$ satisfy the hypothesis (H.6) of~\cite[\S3.5]{mr02}. We may therefore apply the main results of~\cite{mr02}. In particular, the existence of Kolyvagin systems for these Selmer structures will be decided by their \emph{core Selmer ranks} (for a definition, c.f., \cite[Definition 4.1.11 and 5.2.4]{mr02}). Let $\XX(T,\FF)$ denote the core Selmer rank of the Selmer structure $\FF$, for $\FF=\FFc$ or for $\FF=\FF_{\al}$.

\begin{prop}
\label{prop:canonical core selmer rank}
We have $\XX(T,\FFc)=r\,(=[k:\QQ])$.
\end{prop}
\begin{proof}
This follows from~\cite[Theorem 5.2.15]{mr02}, applied with the base field $k$ (instead of $\QQ$; we therefore have $r$ real places instead of one) and using our assumption that $\chi$ is totally odd.
\end{proof}

\begin{prop}
\label{modifiedcorerank}
The core Selmer rank $\XX(T,\FF_{\al})$ of the Selmer structure $\FF_\al$ on $T$ is one.
\end{prop}
\begin{proof}
The proof of this Proposition is identical to the proof of~\cite[Proposition 1.8]{kbbstark}.
\end{proof}

\subsubsection{Kolyvagin systems over $k$} We recall the definition of the (generalized) module of Kolyvagin systems (introduced in~\cite{mr02}) for the Selmer triple $(T,\FF_\al,\PP)$:
\begin{define}[Compare to~\cite{mr02} Definition 3.1.6]
\label{def:KSzp}
Define the (generalized) module of Kolyvagin systems $$\overline{\KS}(T,\FF_\al,\PP):=\varprojlim_{s}\varinjlim_{j} \KS(T/p^sT,\FF_\al,\PP_j),$$ where $\KS(T/p^sT,\FF_\al,\PP_j)$ is the module of Kolyvagin systems for the Selmer structure $\FF_\al$ on the representation $T/p^sT$, as in~\cite[Definition 3.1.3]{mr02}.

We will call an element of $\overline{\KS}(T,\FF_\al,\PP)$ an \emph{$\al$-restricted Kolyvagin system} for $T$.
\end{define}

\begin{rem}
\label{rem:choiceoftransverse} In order to define Kolyvagin systems,
one first needs to define the ``transverse local condition''
(see~\cite[Definition 1.1.6(iv)]{mr02}). In this remark, we briefly
recall this definition. Let $F$ be any local field, and fix once and
for all an abelian extension $F^\prime/F$ which is totally and
tamely ramified, and moreover is a maximal such extension. When
$F=\QQ_\ell$, then there is a natural choice for $F^\prime$, namely
$F^\prime=\QQ_\ell(\pmb{\mu}_\ell)$. In general, we simply fix an
extension $F^\prime$ as above and define the \emph{transverse local
condition} to be
$$H^1_{\textup{tr}}(F,X)=\ker\{H^1(F,X) \lra H^1(F^\prime,X)\},$$
for appropriate quotients $X$ of $T$.

Let $\qq \in \PP_j$ for some $j$ and consider now the case
$F=k_{\qq}$. Starting from \S\ref{sec:ESKSmap}, we will insist that
the extension $F^\prime $ contains $k(\qq)_{\qq}$, where $k(\qq)$ is
the maximal $p$-extension inside the ray class field of $k$ modulo
the prime ideal $\qq$. Although we do not need this assumption for
the results in \S\ref{subsec:KS1}, it is necessary to choose
$F^\prime$ in this manner to be able to modify the arguments
of~\cite[Theorem 3.2.4]{mr02} in order to obtain a proof of
Theorem~\ref{thm:ESKSmain} below.
\end{rem}

\begin{prop}
\label{prop:modifiedKSrank1}
The $\ZZ_p$-module $\overline{\KS}(T,\FF_\al,\PP)$ is free of rank \emph{one}. Furthermore, it is generated by a Kolyvagin system $\kappa \in\overline{\KS}(T,\FF_\al,\PP)$ whose image \textup{(}under the canonical map induced from reduction mod $p$\textup{)} inside $\textup{\textbf{KS}}(T/pT,\FF_{\al},\PP)$ is nonzero.
\end{prop}
A generator of the cyclic module $\overline{\KS}(T,\FF_\al,\PP)$ will be called a~\emph{primitive} Kolyvagin system.
\begin{proof}
This is immediate after Proposition~\ref{modifiedcorerank} and \cite[Theorem 5.2.10]{mr02}. To apply Theorem 5.2.10 of loc.cit., one needs to verify that the hypotheses (H.1)-(H.6)  of~\cite[\S3.5]{mr02} hold true for the triple $(T,\FF_\al,\PP)$.
\end{proof}
\begin{rem}
\label{rem:barvsnobar}
Using  Proposition~\ref{modifiedcorerank} and \cite[Proposition 5.2.9]{mr02}, the generalized module of Kolyvagin systems $\overline{\KS}(T,\FF_\al,\PP)$ may be identified by the module of of Kolyvagin systems ${\KS}(T,\FF_\al,\PP)$ (defined as in~\cite[Definition 3.1.3]{mr02}). We will use this identification without warning.
\end{rem}

We record here the main application of a Kolyvagin system for the Selmer triple $(T,\FF_\al,\PP)$. Suppose $\left\{\{\kappa_\tau(s)\}_{\tau\in \NN_s}\right\}_{_{s}}=\kappa \in\overline{\KS}(T,\FF_\al,\PP)$ is any Kolyvagin system. See~\cite[\S3]{mr02} for an explanation of our notation. We loosely say here that $\kappa_\tau(s) \in H^1(k,T/p^sT)$, and by definition, there is a well defined element $$\kappa_1=\{\kappa_1(s)\}_{s} \in \varprojlim_{s} H^1_{\FF_{\al}}(k,T/p^sT)=H^1_{\FF_{\al}}(k,T).$$

\begin{thm}\cite[Theorem 5.2.13 and 5.2.14]{mr02}
\label{thm:mainksapplicationk}
Under our running hypotheses
\begin{enumerate}
\item[(i)] $\textup{length}(H^1_{\FF_{\al}^*}(k,T^*))\leq \textup{length}(H^1_{\FF_{\al}}(k,T)/\ZZ_p\cdot\kappa_1),$
\item[(ii)] the inequality in \textup{(i)} is an equality if and only if $\kappa$ is primitive.
\end{enumerate}
\end{thm}
\begin{rem}
\label{rem:question}
Note that the \emph{choice} of a rank one direct summand $\al \subset H^1(k_p,T)$ makes our approach somewhat  unnatural. We address this issue in this remark. Put
\be\label{eqn:decomposeH1} H^1(k_p,T)=\bigoplus_{i=1}^{r}\al_i\ee
(where each $\al_i$ is a free $\ZZ_p$-submodule of $H^1(k_p,T)$ of rank one) and consider \be\label{eqn:sum} \sum_{i=1}^{r}\KS(T,\FF_{\al_i},\PP)\subset \KS(T,\FFc,\PP).\ee
\begin{claim}
The sum in \textup{(}\ref{eqn:sum}\textup{)} is a direct sum.
\end{claim}
\begin{proof}
Assume contrary: Suppose $0\neq \pmb{\kappa}^i \in \KS(T,\FF_{\al_i},\PP)$ (for $i=1,\dots,r$) is such that $$\sum_{i=1}^{r}a_i\pmb{\kappa}^i=0$$ for some $a_i \in \ZZ_p$, and $a_{i_0}\neq 0$ for a certain $1\leq i_0\leq r$. This means
\be\label{eqn:linesequality}a_{i_0}\pmb{\kappa}^{i_0}=-\sum_{\substack{i=1\\i\neq i_0}}^{r}a_{i}\pmb{\kappa}^i \in \sum_{\substack{i=1\\i\neq i_0}}^{r} \KS(T,\FF_{\al_i},\PP).\ee
 Write $\pmb{\kappa}^{i_0}=\{\kappa_n^{i_0}\}$ (see \cite[\S3]{mr02} for a precise definition of a Kolyvagin system to clarify this notation, see also Remark~\ref{rem:choiceoftransverse} below). Equation (\ref{eqn:linesequality}) therefore shows that
 \be\label{eqn:wronglines}
 \textup{loc}_p(a_{i_0}{\kappa}^{i_0}_1) \in \bigoplus_{\substack{i=1 \\i \neq i_0}}^{r} \al_i.
 \ee
 Also, by definition, $\textup{loc}_p(a_{i_0}\kappa^{i_0}_1) \in \al_{i_0}$; using this together with~(\ref{eqn:wronglines}) we conclude that $\textup{loc}_p(a_{i_0}\kappa_1^{i_0})=0$. The injectivity of $\textup{loc}_p$ (which we checked in~\S\ref{subsubsec:compareselmer1}) gives $a_{i_0}\kappa_1^{i_0}=0$.

On the other hand, Proposition~\ref{prop:compare selmer} (applied with $\al=\al_{i_0}$) shows that $H^1_{\FF_{\al_{i_0}}^*}(k,T^*)$ is finite (as the finite group $H^1_{\FF_{\textup{cl}}^*}(k,T^*)^\vee=(A_L^{\chi})^{\vee}$ surjects onto its Pontryagin dual). This in return shows, using~\cite[5.2.12(v)]{mr02}, that for any  $0\neq\pmb{\kappa}=\{\kappa_n\} \in \KS(T,\FF_{\al_{i_0}},\PP)$, we have $\kappa_1\neq0$. Therefore, $a_{i_0}\kappa_1^{i_0}=0$ implies that $a_{i_0}\kappa^{i_0}=0$, a contradiction.
\end{proof}
Note that, in order to prove the Claim above,  we used the facts that $\textup{loc}_p$ is injective (on $H^1_{\FF_{\al_{i_0}}}(k,T)$) and that $H^1_{\FF_{\al_{i_0}}^*}(k,T^*)$ is finite in our current setting. With a bit more work, it is possible to prove this Claim without having either of these conditions. We leave the more general proof aside not to digress from the main point of our paper any further.

It would be very interesting to have an answer for the following: \\
\textbf{Question:} Is the direct sum $$\bigoplus_{i=1}^{r}\KS(T,\FF_{\al_i},\PP)\subset \KS(T,\FFc,\PP)$$ independent of the choice of the decomposition (\ref{eqn:decomposeH1})?

When the answer to this question is affirmative, we would have a \emph{canonically} defined rank-$r$ submodule of $\KS(T,\FFc,\PP)$. It would be even more tempting to inquire whether this rank-$r$ submodule descends from Euler systems. In \S3 below, we construct a rank $r$ submodule of $ \KS(T,\FFc,\PP)$ out of Stickelberger elements; which still does depend on the decomposition~(\ref{eqn:decomposeH1}).

These questions seem to be out of reach in the current state of the art. We may, however, prove the following weaker (yet still interesting) statement. First, we recall some terminology from~\cite{mr02}.

Define the module of $L$-values
\begin{align*}\mathbf{LV}=&\,\textbf{LV}(T;\{\al_i\}_{_{i=1}}^{^{r}}):=\\ &\textup{span}_{\ZZ_p}(\kappa_1: \pmb{\kappa}\in\KS(T,\FF_{\al_i},\PP) \hbox{ for some } i) \subset H^1_{\FFc}(k,T).\end{align*} (Compare this definition with~\cite[Definition 3.1.5]{mr02}.) Note that the $\ZZ_p$-module $\mathbf{LV}$ depends a priori on the choice of the decomposition~(\ref{eqn:decomposeH1}).
\begin{thm}
\label{thm:independent}
The module of $L$-values $\mathbf{LV}$ is independent of the choice of the decomposition~\textup{(}\ref{eqn:decomposeH1}\textup{)}.
\end{thm}

\begin{proof}
Fix any generator $\ell_i$ of the free $\ZZ_p$-module $\al_i$ of rank one. Suppose $\al \subset H^1(k_p,T)$ is any rank one direct summand (not necessarily one of $\al_i$ which appear in (\ref{eqn:decomposeH1})). Let $\pmb{\kappa}^{\al}=\{\kappa_n^{\al}\}$ be any generator of the cyclic $\ZZ_p$-module $\KS(T,\FF_{\al},\PP)$. To prove the Theorem, it suffices to show that $\kappa_1^{\al} \in \mathbf{LV}.$ We may write $$\textup{loc}_p(\kappa_1^{\al})=\sum_{i=1}^ra_i\ell_i$$ with $a_i \in \ZZ_p$.
\begin{claim}
Let $a_i$ be as above. Then $$\textup{ord}_p(a_i)\geq \textup{ord}_p(\#H^1_{\FF_{\textup{cl}}^*}(k,T^*))$$ for all $1\leq i \leq r$.
\end{claim}
\begin{proof}[Proof of the Claim] Let $d=\textup{gcd}(a_1,\dots,a_r)$ and set $\alpha_i:=\frac{a_i}{d} \in \ZZ_p$. By definition, at least  one of the $\alpha_i$ is a $p$-adic unit. We also set $$x:=\textup{loc}_p(\kappa_1^{\al})=\sum_i^r a_i\ell_i,\hbox{ and, } y=\frac{x}{d}=\sum_i^r\alpha_i\ell_i.$$
\begin{enumerate}
\item Since $d\cdot y =x \in \al$ and $H^1(k_p,T)/\al$ is $\ZZ_p$-torsion free, it follows that $y \in \al$.
\item $H^1(k_p,T)/\ZZ_p y$ is $\ZZ_p$-torsion free; indeed suppose $$z=\sum_i^r\beta_i\ell_i\in \bigoplus_i^r\al_i=H^1(k_p,T)$$ is such that $m\cdot z \in \ZZ_p y$ for some $m\in\ZZ_p$. This means $$\sum_i^r m\beta_i\ell_i=sy=\sum_i^rs\alpha_i\ell_i $$ for some $s\in\ZZ_p$, hence $m\beta_i=s\alpha_i$, in particular $m|s\alpha_i$ for every $1\leq i\leq r$. Since $\textup{gcd}(\alpha_1,\dots,\alpha_r)=1$, it follows that $m|s$, hence $z=\frac{s}{m}y \in \ZZ_py$.
\item The items (1) and (2) together show that $\al=\ZZ_py$.
\end{enumerate}
We may now conclude that $$\#\al/\ZZ_p\textup{loc}_p(\kappa_1^{\al})=\#\left(\ZZ_p y/\ZZ_p x\right)=p^{\textup{ord}_p(d)},$$ with $d$ as above. On the other hand, Corollary~\ref{cor:compare over k with class} shows that $$\#\frac{\al}{\ZZ_p\textup{loc}_p(\kappa_1^{\al})}\geq \#H^1_{\FF_{\textup{cl}}^*}(k,T^*) \iff \#\frac{H^1_{\FF_{\al}}(k,T)}{\ZZ_p\kappa_1^{\al}}\geq \#H^1_{\FF_{\al}^*}(k,T^*).$$
The latter statement is the main application of the Kolyvagin system $\kappa^{\al}$ (c.f., Theorem~\ref{thm:mainksapplicationk} above). We therefore conclude that $$p^{\textup{ord}_p(d)}\geq \#H^1_{\FF_{\textup{cl}}^*}(k,T^*),$$ which is our Claim.
\end{proof}
We now prove that Theorem~\ref{thm:independent} follows from this Claim. As in the final paragraph of the proof of the Claim above, it follows from Corollary~\ref{cor:compare over k with class} and \cite[Theorem 5.2.10 and 5.2.14]{mr02} that there exists a Kolyvagin system $\tilde{\pmb \kappa}^i \in \KS(T,\FF_{\al_i},\PP)$ such that $$\textup{loc}_p(\tilde{\kappa}^i_1)=\#H^1_{\FF_{\textup{cl}}^*}(k,T^*)\cdot\ell_i \in \al_i,$$ for every $i=1,\dots,r$. By the Claim above, there is  a Kolyvagin system $\pmb{\kappa}^i \in  \KS(T,\FF_{\al_i},\PP)$ such that $\textup{loc}_p(\kappa^i _1)=a_i\ell_i$ (just set $\pmb{\kappa}^i=\frac{a_i}{\#H^1_{\FF_{\textup{cl}}^*}(k,T^*)}\tilde{\pmb \kappa}^{i}$). We therefore have $$\textup{loc}_p(\kappa^{\al}_1)=\sum_i^r\textup{loc}_p(\kappa^{i}_1),$$ and since the map $\textup{loc}_p$ is injective in our setting, it follows that $$\kappa^{\al}_1=\sum_i^r\kappa^{i}_1 \in \textbf{LV},$$ as desired.
\end{proof}
We close our Remark noting that all this discussion applies equally well in the setting of~\cite{kbbstark,kbbiwasawa} as long as we assume Leopoldt's conjecture (i.e., the injectivity of $\textup{loc}_p$ in the setting of loc.cit.).
\end{rem}

\subsubsection{Kolyvagin systems over $k_\infty$}
\label{subsubsec:kolsyslambda1}

We start with the observation that the following versions of the hypotheses \htam\, and \hsez\, of~\cite[\S2.2]{kbb} hold for $T$:
\begin{enumerate}
\item[{}](\htamF)\, $(T\otimes\QQ_p/\ZZ_p) ^{\mathcal{I}_{\lambda}}$ is divisible for every prime $\lambda \nmid p$, $\lambda \subset k$.
\item[{}](\hsezF)\,  $H^0(k_\wp,T^*)=0$ for primes $\wp|p$.
\end{enumerate}

We define a Selmer structure $\FFc^\LL$ on certain quotients of $T\otimes\LL$. This is the Selmer structure $\FFc$ of~\cite[Definition 2.2]{kbb}:
\begin{define}
\label{def:kbbcanonicalselmer}
 Suppose $f \in \LL$ is any \emph{distinguished} polynomial, in the sense that the quotient $\LL/(f)$ is a free $\ZZ_p$-module of finite rank. Let $\FFc^\LL$ be the following  Selmer structure on $T_f:=T \otimes \Lambda /(f)$:
\begin{itemize}
\item  $\Sigma(\FFc^\LL)=\Sigma(\FF_{\LL})$,

\item The local conditions are given by $$H ^{1} _{\FFc^\LL}(k_{\lambda}, T_f)= \left \{
\begin{array}{ccc}
     H ^{1} (k_\lambda \,, T_f)& ,& \hbox{if } \lambda|p   \\\\
          H ^{1}_{\finite}(k _{\lambda},T_f)& ,& \hbox{if }\lambda \in \Sigma(\FFc) \hbox{ and } \lambda \nmid p \\
\end{array}
\right.$$
\end{itemize}
with $$H ^{1}_{\finite} (k_\lambda, T_f) = \hbox{ker} \{ H ^{1} (k_\lambda, T_f) \lra H^{1} (k_\lambda^{\unr} , T _f \otimes \QQ_p) \},$$ where $k_\lambda^{\unr}$ is the maximal unramified extension of $k_\lambda$.

The induced Selmer structure on the quotients $T \otimes \Lambda / (p^s,f)$, which is obtained by \emph{propagating}  $\FFc^\LL$ (in the sense of Definition~\ref{def:propagation}) will also be denoted by $\FFc^\LL$.

\end{define}

Let $T_{s,m}:=T\otimes\LL/(p^s,\xx^m)$, where $\xx$ is as in \S\ref{subsub:lockinfty}.

\begin{rem}
\label{rem:manyselmerstructuresagree}
By the definition of $\FF_{\all}$, the local conditions on $T_{s,m}$ at primes $\lambda\nmid p$ propagated from $\FF_{\all}$ coincide with the local conditions propagated from $\FF_\LL$; and thanks to \cite[Corollary 2.8 and 2.9]{kbb}, they also coincide with the local conditions determined by $\FFc^\LL$, since  \htamF\, holds true. Indeed, it is proved in loc.cit. that all these local conditions coincide with $$H^1_{\unr}(k_\lambda,T_{s,m}):=\ker\left\{H^1(k_\lambda,T_{s,m}) \lra H^1(k_\lambda^{\unr},T_{s,m})\right\},$$ as long as the hypothesis \htamF\, holds true. We note further that $\FF_{\all}$ propagates to the Selmer structure $\FF_\al$ on $T=T\otimes\LL/(\xx)$.
\end{rem}
\begin{define}(Compare with~\cite[Definition 3.1.6]{mr02})
\label{def:KSboth}
We define the module of \emph{$\all$-restricted $\LL$-adic Kolyvagin systems} to be
$$\overline{\KS}(T\otimes\LL,\FF_{\all},\PP):=\varprojlim_{s,m}\varinjlim_{j} \KS(T_{s,m},\FF_{\all},\PP_j),$$ where $\KS(T_{s,m},\FF_{\all},\PP_j)$ is the module of Kolyvagin systems for the Selmer structure $\FF_{\all}$ on the representation $T_{s,m}$.
\end{define}

\begin{thm}
\label{main-stark}
Suppose \htamF\, and \hsezF\, hold true. Then $\LL$-module  $\overline{\KS}(T\otimes\LL,\FF_{\all},\PP)$  is free of rank one, and the canonical map
$$\overline{\KS}(T\otimes\LL,\FF_{\all},\PP) \lra \overline{\KS}(T,\FF_{\mathcal{L}},\PP)$$
 is surjective.
\end{thm}
Note that, thanks to our assumption (A1), both \htamF\, and \hsezF\, are true for the particular Galois representation $T$ we are interested in. Any generator of the cyclic $\Lambda$-module $\overline{\KS}(T\otimes\LL,\FF_{\all},\PP)$ will be called a \emph{primitive} $\LL$-adic Kolyvagin system.
\begin{proof}
The proof of this theorem is very similar to the proof of~\cite[Theorem 2.19]{kbbiwasawa} and we refer the reader to loc.cit. for details. We only remark here that the proof follows from an appropriate variant of~\cite[Theorem 3.23]{kbb}. Note that Theorem 3.23 of loc.cit. applies (with the base field $\QQ$ replaced by $k$, and the Selmer structure $\FFc$ replaced by $\FF_{\all}$) thanks to Proposition~\ref{modifiedcorerank} and the truth of the hypotheses \hone-\hfour, \htamF, \hsezF.

\end{proof}

In \S\ref{subsec:kolsys2} below, we explain how to obtain these Kolyvagin systems out of the Stickelberger elements, assuming a weak version of \emph{Brumer's conjecture.}
Note, however, that the existence of $\LL$-adic Kolyvagin systems does not rely on Brumer's conjecture.

We record here the main application of a $\LL$-adic Kolyvagin system
$$\pmb{\kappa}=\left\{\{\kappa_\tau(s,m)\}_{\tau \in \NN_{s+m}}\right\}_{_{s,m}}.$$
For an explanation of our notation, see~\cite[\S3]{mr02}. Here we only note that $\kappa_\tau(s,m) \in H^1(k,T_{s,m})$, and by definition, there is a well defined element
$$\kappa_1=\{\kappa_1(s,m)\}_{s,m} \in \varprojlim_{s,m} H^1_{\FF_{\all}}(k,T_{s,m})=H^1_{\FF_{\all}}(k,T\otimes\LL).$$

For notational simplicity, we write $\mathbb{T}=T\otimes\LL$. Recall that $\textup{char}(\mathbb{A})$ denotes the characteristic ideal of a finitely generated $\LL$-module $\mathbb{A}$, with the convention that $\textup{char}(\mathbb{A})=0$ unless $\mathbb{A}$ is $\LL$-torsion.
\begin{thm}
\label{thm:mainapplicationKS}
Under the assumptions \textup{(A1)-(A2)},
$$\textup{char}(H^1_{\FF_{\all}^*}(k,\mathbb{T}^*)^\vee)\mid \textup{char}(H^1_{\FF_{\all}}(k,\mathbb{T})/\LL\cdot\kappa_1).$$
\end{thm}
\begin{proof}
This is~\cite[Theorem 5.3.10(iii)]{mr02} applied in our setting. We remark that all the hypotheses of Theorem 5.3.10(iii) of loc.cit. hold thanks to (A1)-(A2) (as we have already demonstrated above).
\end{proof}

 \section{Euler systems from Stickelberger elements}
 \label{sec:ES}

 We begin with recalling the definition of Stickelberger elements. We first set our notation. Assume $k,\chi, \ff=\ff_{\chi}$ and $L$ are as above. For a (square free) cycle $\tau=\qq_1\dots\qq_m$ of the number field $k$, let $k(\tau)$ be the compositum
  $$k(\tau)=k(\qq_1)\cdot\cdot \cdot k(\qq_m),$$
   where $k(\qq)$ denotes the maximal $p$-extension inside the ray class field of $k$ modulo the prime ideal $\qq$. For any field $K$, define $K(\tau)$ as the composite of $k(\tau)$ and $K$. Let
   $$\mathcal{K} =\{L_n(\tau): \tau \in \NN;\,n\geq0\},$$
   $$\mathcal{K}_0 =\{k_n(\tau): \tau \in \NN;\,n\geq0\}$$
    be two collection of abelian extensions  of $k$. Note that any  field $L_n(\tau) \in \mathcal{K}$ is CM and abelian over the totally real field $k$. Let $S$ be the set of places of $k$, consisting of all places above $p$, all places dividing $\ff$ and all infinite places. For any $K \in \kk$, write $S_K$ for the set of all places of the field $K$ lying above the places in $S$. When there is no confusion, we will simply write $S$ for $S_K$.

 For any $K \in \kk$, the partial zeta function for $\sigma \in \Gal(K/k)$ is defined as usual by $$\zeta_S(s,\sigma):=\sum_{\substack{(\frak{a},K/k)=\sigma\\ \frak{a} \textup{ is prime to }S}} \mathbf{N}\frak{a}^{-s}$$
for $\textup{Re}(s)>1$. Here $\mathbf{N}\frak{a}$ is the absolute norm of the ideal $\frak{a} \in k$, and $(\frak{a},K/k)$ is the Artin symbol. The partial zeta functions admit a meromorphic continuation to the whole complex plane, and holomorphic everywhere except at $s=1$. We may therefore set
$$\theta_K={\theta}_{K,S}:=\sum_{\sigma \in \Gal(K/k)}\zeta_S(0,\sigma)\sigma^{-1} \in \mathbb{C}[\Gal(K/k)].$$
Thanks to~\cite{siegel70}, $\theta_K$ is an element of $\QQ[\Gal(K/k)]$. Further, we know for the $\chi$-part $\theta_K^\chi$ of $\theta_K$, thanks to~\cite{deligne-ribet}, that $\theta_K^{\chi} \in \ZZ_p[\Gal(K/k)]^{\chi}$, since we assumed $\chi\neq \omega$.

\begin{lemma}
\label{lem:stickdistribution}
For any $L_n(\tau)=K\subset K^\prime=L_{n^\prime}(\tau^\prime)$ inside $\kk$,
 $$\theta_{K^\prime}|_{_{K}}=\prod_{\substack{\,\qq \mid \tau^{\prime},\, \qq \nmid \tau}}(1-\textup{Fr}_{\qq}^{-1})\theta_K.$$
\end{lemma}

\begin{proof}
This follows from~\cite[Proposition IV.1.8]{tate}.
\end{proof}

As before, let $A_K$ denote the $p$-part of the ideal class group of $K \in \kk$, and $A_K^{\chi}$ its $\chi$-isotypic part. Until the end of this section we suppose that the $\chi$-part of the Brumer's conjecture (Assumption~\ref{assume:brumer}) holds true.

\begin{rem}
Greither~\cite[Corollary 2.7]{greither} and Kurihara~\cite[Corollary 2.4]{kurihara} have deduced Assumption~\ref{assume:brumer} from Iwasawa's main conjecture in this setting (which holds thanks to~\cite{wiles-mainconj}) and the vanishing of the Iwasawa $\mu$-invariant for $K$. However, we wish not to assume the truth of the main conjecture; in fact we rather assume in this paper Assumption~\ref{assume:brumer} and deduce the main conjecture itself.

Having referred  the reader to~\cite{kurihara}, we caution the reader about one minor point: If a prime $\wp \subset k$ above $p$ is unramified in $K/k$, then Kurihara's Stickelberger element $\tilde{\theta}_K^\chi$ differs from our $\theta_K^{\chi}$ by a factor of $(1-\textup{Fr}_\wp)^{\chi}$, where $\textup{Fr}_\wp$ is the Frobenius at $\wp$ for the unramified extension $K/k$. If (A1) holds, it follows that  $(1-\textup{Fr}_\wp)^{\chi}$ is a unit inside $\ZZ_p[\Gal(K/k)]^{\chi}$. Therefore, the statement of Assumption~\ref{assume:brumer} is still equivalent to the statement $\tilde{\theta}_K^\chi\cdot A_K^\chi=0$, which is the assertion deduced from the main conjecture in~\cite{kurihara}.
\end{rem}

Suppose $F$ is any finite abelian extension of $k$, and $K=FL$. Then by the  inflation-restriction sequence and class field theory one has
\be
\label{eqn:explicith1forzp1}
H^1(F,\ZZ_p(\chi))\cong H^1(K,\ZZ_p)^{\chi^{-1}}=\Hom(\mathbb{A}_K^{\times}/K^\times,\ZZ_p)^{\chi^{-1}},
\ee
where $\mathbb{A}_K^{\times}$ denotes the ideles of $K$. Since any continuous homomorphism of $\mathbb{A}^\times_K$ into $\ZZ_p$ should vanish on $$B_K:=\prod_{w|\infty}K_w^{\times}\times\prod_{w|p}\{1\}\times\prod_{w\nmid p\infty} \mathcal{O}_{K_w}^{\times}\subset \mathbb{A}_K^{\times},$$
Equation (\ref{eqn:explicith1forzp1}) gives

\be
\label{eqn:explicith1forzp2}
{H^1(F,\ZZ_p(\chi))}\cong\Hom({\mathbb{A}_K^{\times}/K^\times B_K},\ZZ_p)^{\chi^{-1}}=\Hom({\left( \mathbb{A}^{\times}_K/K^{\times} B_K\right)}^{\chi},\ZZ_p).
\ee
Further, there is an exact sequence
$$0\lra U_K/\overline{\mathcal{O}_K^{\times}} \lra \mathbb{A}_K^{\times}/K^\times B_K \lra A_K \lra 0,$$
 which is induced from the map that sends an id\`ele to the corresponding ideal class. Here $\overline{\mathcal{O}_K^{\times}}$ is the closure of the global units ${\mathcal{O}_K^{\times}}$ inside the local units $U_K\subset K\otimes\QQ_p$. Since taking $\chi$-parts is exact (as the order of $\chi$ is prime to $p$), we obtain an exact sequence
$$0\lra U_K^{\chi}/(\overline{\mathcal{O}_K^{\times}})^{\chi} \lra (\mathbb{A}_K^{\times}/K^\times B_K)^{\chi} \lra A_K^\chi \lra 0.$$
Thus, by Assumption~\ref{assume:brumer}, multiplication by $\theta_K^{\chi}$ gives a map
$$(\mathbb{A}_K^{\times}/K^\times B_K)^{\chi}\stackrel{\theta_K^{\chi}}{\lra}U_K^{\chi}/(\overline{\mathcal{O}_K^{\times}})^{\chi}.$$
Since we assumed $\chi$ is totally odd, $(\overline{\mathcal{O}_K^{\times}})^{\chi}$ is finite (see the final paragraph of the proof of Proposition~\ref{prop:classical selmer explicit}), and we therefore have an induced map
\be
\label{eqn:themaptolocalunits}
(\mathbb{A}_K^{\times}/K^\times B_K)^{\chi}\stackrel{\theta_K^{\chi}}{\lra}U_K^{\chi}/(U_K^{\chi})_{\textup{tors}}.\ee

Suppose we are given a collection of homomorphisms $\pmb{\lambda}=\{\lambda_n^\tau\}$ with  $\lambda_n^\tau \in \Hom(U_{L_n(\tau)}^\chi,\ZZ_p)$ which satisfies the following properties:
\begin{enumerate}
\item For all $L_n(\tau), L_{n^\prime}(\tau\qq) \in \kk$, the following diagram commutes:
$$\xymatrix@C=2cm @R=.3cm{U_{L_{n^\prime}(\tau\qq)}^\chi \ar[rd]^{\lambda_{n^\prime}^{\tau\qq}}& \\
&\ZZ_p\\
U_{L_n(\tau)}^\chi \ar@{^{(}->}[uu]^{-\textup{Fr}_{\qq}}\ar[ur]_{\lambda_n^\tau}&
}$$
\item For $n^\prime \geq n$, we have $\lambda_{n^\prime}^\tau|_{_{U_{L_n(\tau)}^\chi}}=\lambda_n^\tau$.
\end{enumerate}
 Define
 $$\tilde{c}_{k_n(\tau)} \in \Hom\left(\left(\mathbb{A}_{L_n(\tau)}^{\times}/(L_n(\tau))^\times B_{L_n(\tau)}\right)^{\chi},\ZZ_p\right)$$
  (which we view also as an element of $H^1(k_n(\tau),T)$ via the identification~(\ref{eqn:explicith1forzp2}) above) as the composition\footnote{We remark that any homomorphism $\lambda \in \Hom(U_{L_n(\tau)}^{\chi},\ZZ_p)$ necessarily factors through the quotient $U_{L_n(\tau)}^\chi{/}(U_{L_n(\tau)})_{\textup{tors}}^\chi$; this is how we make sense of the right most map in~(\ref{eqn:star}).}
\be\label{eqn:star}
\tilde{c}_{k_n(\tau)} :\,\left(\mathbb{A}_{L_n(\tau)}^{\times}/(L_n(\tau))^\times B_{L_n(\tau)}\right)^{\chi} \stackrel{\theta_{L_n(\tau)}^{\chi}}{\lra} U_{L_n(\tau)}^\chi/(U_{L_n(\tau)})_{\textup{tors}}^\chi\stackrel{\lambda_n^\tau}{\lra} \ZZ_p.
\ee
Set $\tilde{\mathbf{c}}=\{\tilde{c}_{k_n(\tau)}\}$.

\begin{thm}
\label{thm:mainstickes}
Fix a collection of homomorphisms $\pmb{\lambda}=\{\lambda_n^\tau\}$ as above (if it exists). Then there is an Euler system $\mathbf{c}=\{c_{k_n(\tau)}\}$ \textup{(}which depends on the choice of $\pmb{\lambda}$\textup{)} for the Galois representation $T$ \textup{(}in the sense of~\cite[Definition II.1.1 and Remark II.1.4]{r00}
\textup{)} such that $c_{k_n}=\tilde{c}_{k_n}$ for all $n$.
\end{thm}
In \S\ref{subsec:choosehoms} below, we construct a collection
$\pmb{\lambda}$ which satisfies the desired properties, and hence
conclude with the existence of an Euler system for $T$, assuming the
truth of Assumption 1.1. When $k=\QQ$, this Euler system has been
given by Rubin~\cite[\S3.4]{r00}.
\begin{proof}
Since the proof of this theorem very closely follows the proof of~\cite[Proposition III.3.4]{r00}, we only give a sketch. All the references in this proof are to~\cite{r00}. First, one checks (mimicking the proof of Proposition III.3.4) that the collection $\tilde{\mathbf{c}}$ (which should be compared with the collection $\tilde{\mathbf{c}}^\prime$ of Rubin) satisfies a distribution relation with \emph{wrong} Euler factors. This could be remedied, as in the paragraph following
Remark III.4.4, using Lemma IX.6.1 to obtain a new collection $\mathbf{c}$ (which corresponds to what Rubin calls $\tilde{\mathbf{c}}$) as desired.
\end{proof}

We close this section with a final remark which we will refer to in what follows:
\begin{rem}
\label{rem:semi local identification}
The argument of Remark~\ref{rem:alternative proof} shows that, under the hypothesis (A1),  $$H^1(k_n(\tau)_p,T)\cong \Hom(U_{L_n(\tau)}^{\chi},\ZZ_p).$$
\end{rem}
\begin{rem}
\label{rem:comparetorubin}
In this remark we discuss the main differences between the cases when the base field $k$ is a general totally real field (i.e., the case we study in this article) and the particular case $k=\QQ$ (i.e., the case Rubin studies in~\cite[\S III.4]{r00}).
\begin{itemize}
\item[(i)] The first difference is in regard of the core Selmer ranks: The core rank $\XX(T,\FFc)$ of the canonical Selmer structure (c.f., Example~\ref{example:canonical selmer} above) of the $G_k$-representation $T$  is $[k:\QQ]=r$. Rubin treats the case $r=1$ (i.e., the case $k=\QQ$). In this paper, we study the case $r>1$, adapting the work of Mazur and Rubin~\cite{mr02} to the general case when $\XX(T,\FFc)>1$ via what we call $\pmb{\al}$-restricted Euler systems. Although Kurihara~\cite{kurihara} successfully applies the classical Euler system argument to Stickelberger elements to prove Theorem A and Theorem B above, our approach via $\pmb{\al}$-restricted Euler systems yields in addition a comparison between the Stickelberger elements and Rubin-Stark elements (see Theorem~\ref{thm:maincomparison} below). Furthermore, our approach here fits well into the framework developed in~\cite{mr02} which was later enhanced by the author in~\cite{kbbstark, kbbiwasawa, kbbrankr}.
\item[(ii)] The second difference is the manner the collection $\pmb{\lambda}=\{\lambda_n^\tau\}$ of homomorphisms is chosen. In~\cite[Appendix D]{r00}, Rubin constructs these homomorphisms explicitly when $k=\QQ$. This construction is not available when $k\neq \QQ$; that is why we prove ``abstractly'' in~\S\ref{subsec:choosehoms} that a collection $\pmb{\lambda}$ exists with the desired properties.
\end{itemize}
\end{rem}

 \section{Euler systems to Kolyvagin systems map}
\label{sec:ESKSmap}
We first recall what Mazur and Rubin call the \emph{Euler system to Kolyvagin system map}. Suppose $T$, $\kk$ and $\PP$ are as above. Let $\ES(T)=\ES(T,\kk)$ denote the collection of Euler systems for $(T,\kk)$ in the sense of~\cite[\S3]{r00}. Recall also the generalized module of Kolyvagin systems $\overline{\KS}(T,\FF,\PP)$ and $\overline{\KS}(T\otimes\LL,\FF,\PP)$ for various choice of Selmer structures $\FF$.
\begin{thm}[Mazur-Rubin]
\label{thm:ESKSmain} There are canonical maps
\begin{itemize}
\item $\ES(T) \lra\overline{\KS}(T,\FFc,\PP)$,
\item $\ES(T) \lra \overline{\KS}(T\otimes\LL,\FF_{\LL},\PP)$
\end{itemize}
with the following properties:
\begin{enumerate}
\item If $\mathbf{c} \in \ES(T)$ maps to $\pmb{\kappa} \in \overline{\KS}(T,\FFc,\PP)$, then $\kappa_1=c_k$,
\item If $\mathbf{c} \in \ES(T)$ maps to $\pmb{\kappa} \in\overline{\KS}(T\otimes\LL,\FF_\LL,\PP)$, then $$\kappa_1=\{c_{k_n}\} \in \varprojlim_n H^1(k_n,T)=H^1(k,T\otimes\LL).$$
\end{enumerate}
\end{thm}
\begin{proof}
Let $\rho_{\textup{cyc}}: G_k \ra \ZZ_p^\times$ be the cyclotomic
character (giving the action of $G_k$ on $\pmb{\mu}_{p^\infty}$),
and set
$$\rho=\omega^{-1}\rho_{\textup{cyc}}: \Gamma \lra 1+p\ZZ_p.$$ Let
$\psi=\omega\chi^{-1}$ be as in the introduction, and set
$$T^\prime=\ZZ_p(1)\otimes\psi^{-1}=T \otimes \rho, \hbox{ and } \mathbb{T}^\prime =T^\prime \otimes \LL.$$
Note that we have an isomorphism of $G_k$-modules
\be\label{eqn:lamdatwist}\mathbb{T}^\prime=T^\prime\otimes\LL\,
\stackrel{{\otimes\rho^{-1}}}{\lra}\,T\otimes\LL =\mathbb{T},\ee as
$\rho$ is a character of $\Gamma$. We then have the following
diagram:
$$\xymatrix{\textup{ES}(T)\ar[r]^{\tau_\rho} \ar@{-->}[rd] \ar@{-->}@/_5.8pc/[rrrd]&\textup{ES}(T^\prime) \ar[r]^(.42){\partial^{\prime}}& \overline{\textup{KS}}({T}^{\prime}\otimes\LL,\FF_{\LL})\ar[dl]^{\tau_{\rho^{-1}}}\ar[dr]^{\frak{s}^{\prime}}\\
&
\overline{\textup{KS}}({T\otimes\LL},\FF_{\LL})\ar[rr]_{\frak{s}}&&
\overline{\textup{KS}}({T},\FFc)}$$ \vskip1.2cm The two dotted
arrows are the maps claimed to exist in the statement of the
Theorem, and they are given as the composition of relevant maps in
the diagram. We now explain how the other arrows are obtained. The
map $\tau_\rho$ is obtained by applying a formal twisting argument,
see~\cite[\S6]{r00}. The map $\tau_{\rho^{-1}}$ is induced from the
isomorphism~(\ref{eqn:lamdatwist}), and $\frak{s}$ is induced from
the specialization $T\otimes\LL \ra T$ whose kernel is the
augmentation ideal of $\LL$. Similarly, $\frak{s}^{\prime}$ is
induced from the specialization $T^{\prime}\otimes\LL \ra T$ which
makes the triangle
$$\xymatrix{T^{\prime}\otimes\LL \ar[rd]_{\frak{s}^{\prime}}\ar[rr]^{\otimes\rho^{-1}}&&
T\otimes\LL\ar[ld]^{\frak{s}}\\
&T&}$$ (as well as the triangle in the diagram above) commutative.
The map $\partial^{\prime}$ is the Euler systems to Kolyvagin
systems map of Mazur and Rubin~\cite[Theorem 5.3.3]{r00}, which is
obtained as follows. Starting with an Euler system $\pmb{c}^{\prime}
\in \textup{ES}(T^\prime)$ for $T^{\prime}$ , Kolyvagin's
construction (c.f., \cite[\S4]{r00}) yields a weak Kolyvagin system
(in the sense of~\cite[Definition 3.1.8]{mr02})
$$\pmb{\kappa}^{\textup{w}}=\{\kappa_{\eta}^{\textup{w}}\}_{\eta \in \NN}.$$
The classes $\kappa_{\eta}^{\textup{w}}$ does not necessarily
satisfy the transverse local condition at a prime $\lambda\mid
\eta$. One may however first calculate the finite projections of
these classes (slightly modifying (by replacing $\QQ(\ell)$ by
$k(\lambda)$ and $\QQ(n)$ by $k(\eta)$ where necessary) as in
Theorem A.4 (see particularly Lemma A.6 and Proposition A.8 for the
key steps) of~\cite{mr02}). Note that we pass to an auxiliary twist
$T^\prime$ to ensure that $\textup{Fr}_{\lambda}^{p^m}-1$ acts
injectively on $T^{\prime}$ for every $\lambda\in \PP$ and for every
$m\in \ZZ^{+}$, which is needed for the arguments of Mazur and Rubin
to carry out. Finally, one may modify $\pmb{\kappa}^{\textup{w}}$,
as Mazur and Rubin does in the displayed equation (33)
of~\cite{mr02} to kill its finite projections and thus obtain a
Kolyvagin system, as desired.
\end{proof}
\begin{rem}
Mazur and Rubin's definition of the generalized module of $\LL$-adic Kolyvagin systems $\overline{\KS}(T\otimes\LL,\FF_\LL,\PP)$ slightly differs from our definition of this module (Definition~\ref{def:KSboth}). It is not hard to see that these two definitions give rise to isomorphic modules; see also Remark~\ref{1-kolsys} below.
\end{rem}
We would like to apply this map on the Euler systems we have constructed\footnote{Modulo the existence of a family of homomorphisms $\pmb{\lambda}$.} in~\S\ref{sec:ES}. Note however that Theorem~\ref{thm:ESKSmain} will give rise to Kolyvagin systems only for the coarser Selmer structures $\FF_\LL$ and $\FFc$ (rather than finer Selmer structures $\FF_\all$ and $\FF_\al$). To be able to obtain Kolyvagin systems for the modified Selmer structures $\FF_\all$ and $\FF_\al$, we need to analyze the structure of the semi-local cohomology groups for $T\otimes\LL$ and $T$, over various ray class fields of $k$. This is carried out in \S\ref{subsec:choosehoms}. We then apply the results of \S\ref{subsec:choosehoms} to construct the desired Kolyvagin systems for the modified Selmer structures in \S\ref{subsec:kolsys2}.

\begin{rem}
\label{rem:remark}
In effect, one only needs a \emph{weak Kolyvagin system} (in the sense of~\cite[Definition 3.1.8]{mr02}) for the main application of the Euler system/Kolyvagin system machinery, i.e., for bounding the dual Selmer group. Weak Kolyvagin systems are essentially the derivative classes of Kolyvagin (cf. \cite[\S IV]{r00}) which are obtained directly applying the derivative operators, without the need of the alterations carried out in~\cite[Appendix A]{mr02}.
\end{rem}

 \subsection{A good choice of homomorphisms}
 \label{subsec:choosehoms}
 Recall that $k_\infty$ is the cyclotomic $\ZZ_p$-extension of $k$, and $\Gamma=\Gal(k_\infty/k)$. Let $k_n$ denote the unique sub-extension of $k_\infty/k$ with $[k_n:k]=p^n$ and set $\Gamma_n:=\Gal(k_n/k)$. Recall also that $\Delta_\tau:=\Gal(k(\tau)/k)$.
 \begin{lemma}
 \label{lem:surj}
For every $n\in\ZZ_{\geq 0}$ and $\tau\in \NN(\PP)$, the following corestriction maps on the semi-local cohomology
  \begin{enumerate}
 \item[(i)] $H^1(k_n(\tau)_p,T)\lra H^1(k(\tau)_p,T)$,\\
 \item[(ii)] $H^1(k(\tau)_p,T)\lra H^1(k_p,T)$,\\
 \item[(iii)] $H^1(k_n(\tau)_p,T)\lra H^1(k_p,T)$
  \end{enumerate}
are surjective.
 \end{lemma}
 \begin{proof}
 The cokernel of the map $$H^1(k(\tau),T\otimes\LL)=\varprojlim_n H^1(k_n(\tau)_p,T)\lra H^1(k(\tau)_p,T)$$ is given by $H^2(k(\tau)_p,T\otimes\LL)[\gamma-1]$, where $\gamma$ is any topological generator of $\Gamma=\Gal(k_\infty/k)$. Since it is known that~(cf. \cite{pr}) $H^2(k(\tau)_p,T\otimes\LL)$ is a finitely generated $\ZZ_p$-module, it follows that $$H^2(k(\tau)_p,T\otimes\LL)[\gamma-1]=0 \iff H^2(k(\tau)_p,T\otimes\LL)/(\gamma-1)=0.$$ Since the cohomological dimension of the absolute Galois group of any local field is 2, $$H^2(k(\tau)_p,T\otimes\LL)/(\gamma-1)\cong H^2(k(\tau)_p,T\otimes\LL/(\gamma-1))=H^2(k(\tau)_p,T).$$ It therefore suffices to check that $$H^2(k(\tau)_p,T):=\bigoplus_{v|p}H^2(k(\tau)_v,T)=0,$$ which, via local duality is equivalent to checking that $(T^*)^{G_{k(\tau)_v}}=0$ for each $v|p$.

Write $\mathcal{D}_v$ for the decomposition group at $v\mid p$ inside $\Gal(k(\tau)/k):=\Delta_\tau$. We may identify $\mathcal{D}_v \subset \Delta_\tau$ by the local Galois group $\Gal(k(\tau)_v/k_\wp)$ where $\wp\subset k$ is the prime below $v$. Since $\Delta_\tau$ is generated by inertia groups at the primes dividing $\tau$, all of which act trivially on $T^*$ (by the choice of $\tau$'s). Hence, it follows that $$(T^*)^{G_{k(\tau)_v}}=(T^*)^{G_{k_{\wp}}}.$$ Note that $T^*=\pmb{\mu}_{p^{\infty}}\otimes\chi^{-1}$, so it follows at once that
$(T^*)^{G_{k_{\wp}}}=0$, and thus (i) is proved.

Set $T_\tau:=\textup{Ind}_{k(\tau)}^{k}T$. The semi-local version of Shapiro's lemma (which is explained in~\cite[\S A.5]{r00}) shows that
$$H^1(k(\tau)_p, T) \cong H^1(k_p,T_\tau).$$
The corestriction map $$\mathbf{N}_\tau: \,H^1(k_p,T_\tau)\cong H^1(k(\tau)_p,T)\lra H^1(k_p,T)$$ is simply induced from the augmentation sequence $$0\lra \mathcal{A}_\tau\cdot T_\tau\lra T_\tau\lra T\lra 0,$$ where $\mathcal{A}_\tau$ is the augmentation ideal of the local ring $\ZZ_p[\Delta_\tau]$. The argument above shows that the cokernel of $\bf{N}_\tau$ is dual to $H^0(k_p,(\mathcal{A}_\tau\cdot T_\tau)^*).$ Furthermore, $$(\mathcal{A}_\tau\cdot T_\tau)^{*}: =\Hom(\mathcal{A}_\tau\cdot T_\tau,\pmb{\mu}_{p^{\infty}})=\Hom(\mathcal{A}_\tau\cdot T_\tau,\QQ_p/\ZZ_p)\otimes\ZZ_p(1),$$ and $\Hom(\mathcal{A}_\tau\cdot T_\tau,\QQ_p/\ZZ_p)=\mathcal{A}_\tau\cdot\Hom(T_\tau,\QQ_p/\ZZ_p)$, we thence see that
$$H^0(k_p,(\mathcal{A}_\tau\cdot T_\tau)^*) \hookrightarrow H^0(k_p,T_\tau^{*}).$$
It therefore suffices to show that $H^0(k_p,T_\tau^*)=0$. By local duality, this is equivalent to proving $H^2(k_p,T_\tau)=0$, which, by the semi-local version of Shapiro's lemma, is equivalent to checking $H^2(k(\tau)_p,T)=0$. This final statement is equivalent to the assertion that $H^0(k(\tau)_p,T^*)=0$ by local duality. This, however, has been already verified in the third paragraph of this proof. This completes the proof of (ii).

(iii) clearly follows from (i) and (ii).
  \end{proof}
 \begin{prop}
 \label{prop:semilocalstructure}
  For every $\tau\in \NN(\PP)$,
 \begin{enumerate}
\item[(i)]  the semi-local cohomology group $H^1(k(\tau)_p,T)$ is a free $\ZZ_p[\Delta_\tau]$-module of rank $r$,
 \item[(ii)] for every $n\in\ZZ_{\geq 0}$, the $\ZZ_p[\Gamma_n\times\Delta_\tau]$-module $H^1(k_n(\tau)_p,T)$ is free of rank $r$.
 \end{enumerate}
 \end{prop}

 \begin{proof}
 We start with the remark that $H^1(k(\tau)_p,T)$ is a free $\ZZ_p$-module of rank $r\cdot|\Delta_\tau|$. Indeed, this may be proved by the argument of Lemma~\ref{lemma:free for k} (or alternatively, and more directly, following the argument of Remark~\ref{rem:alternative proof}). Further, we know thanks to Lemma~\ref{lem:surj} that the map $$H^1(k(\tau)_p,T)\lra H^1(k_p,T)$$(which could be thought of as the reduction modulo the augmentation ideal $\mathcal{A}_\tau$) is surjective. Nakayama's lemma and Lemma~\ref{lemma:free for k} therefore imply that $H^1(k(\tau)_p,T)$ is generated by (at most) $r$ elements over $\ZZ_p[\Delta_\tau]$. Let $\frak{B}=\{x_1, x_2,\dots,x_r\}$ be any set of such generators. To prove (i), it suffices to check that the $x_i$'s do not admit any $\ZZ_p[\Delta_\tau]$-linear relation. Assume contrary, and suppose there is a relation
 \be\label{eqn:relation}\sum_{i=1}^{r}\alpha_i x_i=0, \,\,\, \alpha_i \in \ZZ_p[\Delta_\tau].
 \ee
 Define
 $$S=\{\delta x_j: \delta \in \Delta_\tau, 1\leq j\leq r\},$$
 and note that $S$ generates (as a $\ZZ_p$-module) $H^1(k(\tau)_p,T)$ by our assumption on $\frak{B}$, and also that  $|S|=r\cdot|\Delta_\tau|=\textup{rank}_{\ZZ_p}(H^1(k(\tau)_p,T)).$ Equation (\ref{eqn:relation}) can be rewritten as $$\sum_{\delta,j} a_{\delta,j}\cdot \delta x_j=0$$ with $a_{\delta,j} \in \ZZ_p$. Since we already know that $H^1(k(\tau)_p,T)$ is $\ZZ_p$-torsion free, we may assume without loss of generality that $a_{\delta_0,j_0} \in \ZZ_p^\times$ for some $\delta_0,j_0$. This in return implies that $$\delta_0x_{j_0} \in \textup{span}_{\ZZ_p}(S-\{\delta_0x_{j_0}\}),$$ hence $H^1(k(\tau)_p,T)$ is generated by $S-\{\delta_0x_{j_0}\}$. This, however, is a contradiction since we already know that the $\ZZ_p$-rank of $H^1(k(\tau)_p,T)$ is $r\cdot|\Delta_\tau|=|S|$, hence it cannot be generated by $|S|-1$ elements over $\ZZ_p$. The proof of (i) now follows.

 One proves (ii)  in an identical fashion, now considering the \emph{augmentation map}
 $$H^1(k_n(\tau)_p,T)\lra H^1(k(\tau)_p,T),$$
  which is surjective thanks to Lemma~\ref{lem:surj}.
 \end{proof}

 Define the field $\frak{F}$ as the compositum of the fields $k(\tau)$, $$\frak{F}=\bigcup_{\tau \in \NN(\PP)}k(\tau),$$ as $\tau$ runs through the set $\NN$. We set $\pmb{\Delta}:=\Gal(\frak{F}/k)$.
 \begin{cor}
 \label{cor:thick free}
 The $\ZZ_p[[\Gamma\times\pmb{\Delta}]]$-module $\varprojlim_{n,\tau} H^1(k_n(\tau)_p,T)$ is free of rank $r$, and the natural projection $$\varprojlim_{n,\tau} H^1(k_n(\tau)_p,T) \lra H^1(k_{m}(\eta)_p,T)$$ is surjective for every $m \in \ZZ_{\geq 0}$ and $\eta \in \NN$.
 \end{cor}
 \begin{proof}
 Immediate after Proposition~\ref{prop:semilocalstructure}.
 \end{proof}

 \begin{define}
 \label{def:thick line}
 Fix a $\ZZ_p[[\Gamma\times\pmb{\Delta}]]$-rank-one  direct summand $\pmb{\mathcal{L}}$ of $\varprojlim_{n,\tau} H^1(k_n(\tau)_p,T)$. Denote its image under the (surjective) map $$\varprojlim_{n,\tau} H^1(k_n(\tau)_p,T) \lra H^1(k_m(\eta)_p,T)$$ by $\al_m^{\eta}$. When $\eta=1$, we simply write $\al_m$ instead of $\al_m^1$; and when $m=0$ we write $\al$ for $\al_0$. Finally, let $\all$ denote the image of $\pmb{\mathcal{L}}$ under the projection $$\varprojlim_{n,\tau} H^1(k_n(\tau)_p,T) \lra \varprojlim_{n} H^1((k_n)_p,T) =H^1(k_p,T\otimes\LL).$$ We fix generators $\pmb{\varphi}, \varphi_m^{\eta}, \varphi_m, \varphi$ and $\Phi$ of $\pmb{\al}, \al_m^\eta,\al_m, \al$ and $\all$, respectively; such that
 $$\pmb{\varphi} \mapsto \varphi_m^\eta \mapsto \varphi_m,  \hbox{ and}$$
 $$\pmb{\varphi} \mapsto \Phi \mapsto \varphi$$
 under the projection maps we mentioned above.
 \end{define}

 As in Definition~\ref{def:thick line}, we could start with a choice of $\pmb{\al}$, which in return fixes $\all$ and $\al$. Alternatively, we could start with an arbitrary $\al$ (and $\all$) as we did in \S\ref{sub:locp} and show (using linear algebra) that there is a rank one direct summand $\pmb{\al} \subset \varprojlim_{n,\tau} H^1(k_n(\tau)_p,T)$ which projects down to $\al$ (and $\all$), as in   Definition~\ref{def:thick line}.

\begin{rem}
\label{rem:differentLs}
In Definition~\ref{def:KSzp} (resp., Definition~\ref{def:KSboth}) above, we give a definition of a $\al$-restricted Kolyvagin system (resp., $\all$-restricted $\LL$-adic Kolyvagin system). When $\pmb{\al}$ is above, so that the line $\pmb{\al}$ projects down to $\al$ (and respectively, down to $\all$), we also call these \emph{$\pmb{\al}$-restricted Kolyvagin systems.}
\end{rem}
 By Remark~\ref{rem:semi local identification}, we may identify $\varprojlim_{n,\tau} H^1(k_n(\tau)_p,T)$ by the module $\varprojlim_{n,\tau}\Hom\left(U_{L_n(\tau)}^\chi,\ZZ_p\right)$, where we recall that $U_{L_n(\tau)}$ stands for the local units inside $L_n(\tau)\otimes\QQ_p$. We define, for each $m\geq0$ and $\eta \in \NN$, a homomorphism  $\lambda_{m}^{\eta} \in \Hom\left(U_{L_m(\eta)}^\chi,\ZZ_p\right)$ as the composite $$\lambda_{m}^{\eta}:=\varphi_m^{\eta}\circ\prod_{\qq|\eta}(-\textup{Fr}_{\qq}).$$
  We further set $\gimel_m^{\eta}$ for the (free of rank one) $\ZZ_p[\Gamma_m\times\Delta_\eta]$-module generated by $\lambda_{m}^{\eta}$. Clearly,
   \be\label{eqn:modifiedhomsrelation}\gimel_m^\eta=\prod_{\qq|\eta}(-\textup{Fr}_{\qq}^{-1})\al^{\eta}_m=\al_m^\eta,
   \ee
   where the final equality is because $\al_m^\eta$ is a $\ZZ_p[\Gamma_m\times\Delta_\eta]$-stable submodule of
   \be
   \label{eqn:identifywithhomsonunits}
   H^1(k_m(\eta),T)\cong \Hom\left( U_{L_m(\eta)}^\chi,\ZZ_p\right).
   \ee

  When $\eta$ is fixed and $m$ varies, note that the collection $\{\lambda_{m}^{\eta}\}_{_m}$ forms a projective system with respect to norm maps\footnote{Under the identification (\ref{eqn:identifywithhomsonunits}) above, the norm maps on the cohomology are induced from the inclusions  $U_{L_m(\eta)}^\chi \hookrightarrow U_{L_{m^{\prime}}(\eta)}^\chi$, $m^\prime\geq m$.}. When $\eta=1$, we write $\lambda_m$ (resp., $\gimel_m$) instead of $\lambda_m^\eta$ (resp., $\gimel_m^\eta$). Also when $m=0$, we simply write $\lambda$ (resp., $\gimel$) for $\lambda_0$ (resp., $\gimel_0$).

  We finally remark that $\lambda_m=\varphi_m$ for all $m$, by definition.
  \begin{prop}
  \label{prop:constructhoms}
For $\eta, \eta\qq \in \NN$,  and any $m^\prime \geq m$,
\begin{itemize}
\item[(i)]$\lambda_{m^{\prime}}^{\eta\qq}|_{_{U_{L_m(\eta)}^\chi}}=\lambda_m^\eta \circ(-\textup{Fr}_{\qq}),$
\item[(ii)]$\lambda_{m^{\prime}}^{\eta}|_{_{U_{L_m(\eta)}^\chi}}=\lambda_m^\eta.$
\end{itemize}
  \end{prop}

  \begin{proof}
  This is evident, since by construction
  \begin{align*}
  \lambda_{m^{\prime}}^{\eta\qq}|_{_{U_{L_m(\eta)}^\chi}}&=\varphi_{m^{\prime}}^{\eta\qq} \circ\prod_{\varpi|\eta}(-\textup{Fr}_\varpi)(-\textup{Fr}_{\qq})|_{_{U_{L_m(\eta)}^\chi}}\\
  &=\varphi_{m}^{\eta} \circ\prod_{\varpi|\eta}(-\textup{Fr}_\varpi)(-\textup{Fr}_{\qq})|_{_{U_{L_m(\eta)}^\chi}}\\
  &=\lambda_{m}^{\eta} \circ(-\textup{Fr}_{\qq})|_{_{U_{L_m(\eta)}^\chi}},
    \end{align*} where the second equality is because $$\varphi_{m^\prime}^{\eta\qq} |_{_{U_{L_m(\eta)}^\chi}}=\varphi_{m}^{\eta} |_{_{U_{L_m(\eta)}^\chi}}$$ by the norm coherence property of the collection $\{\varphi_{m}^{\eta}\}_{_{\eta,m}}$. This completes (i), and (ii) is proved similarly.

  \end{proof}

  Let $\textbf{c}^{\textup{St}}=\left\{c^{\textup{St}}_{k_n(\tau)}\right\} \in \ES(T)$ be the Euler system constructed via Theorem~\ref{thm:mainstickes} using the Stickelberger elements and the collection $\{\lambda_m^\eta\}$ we defined above. In the next section, we will use $\textbf{c}^{\textup{St}}$ to construct a  Kolyvagin system for the Selmer triple $(T,\FF_{\al},P)$ (resp., for the triple $(T\otimes\LL,\FF_{\all},\PP)$).

 \begin{rmkdef}
 \label{rmkdef:XrestrictedES}
 Let $M$ be any $G_k$-representation which is free of finite rank as a $\ZZ_p$-module and which is unramified outside a finite set of places of $k$. Let $\kk_M$ be a large abelian extension of $k$ defined as in~\cite[Definition 1.1]{r00}. Suppose $\frak{S} \subset \varprojlim_{K\subset\kk_M} H^1(K_p,M)$ is any submodule. Let $\frak{S}_K \subset H^1(K_p,M)$ denote the image of $\frak{S}$ under the obvious projection map. We say that an Euler system  $$\{c_K\}_{K\subset\kk_M}=\mathbf{c} \in \ES(M,\kk_M)=\ES(M)$$ is \emph{$\frak{S}$-restricted} if $\textup{loc}_p(c_K) \in \frak{S}_K$ for any finite extension $K\subset \kk_M$ of $k$. The collection of $\frak{S}$-restricted Euler systems for the pair $(M,\kk_M)$ will be denoted by $\ES_{\frak{S}}(M)=\ES_{\frak{S}}(M,\kk_M)$.
  \end{rmkdef}
 \begin{example}
 \label{ex:Lrestricted}
 \begin{enumerate}
 \item Let $\pmb{\al} \subset \varprojlim_{n,\tau}H^1(k_n(\tau),T)$  be as in Definition~\ref{def:thick line}. It is easy to see that the Euler system $\mathbf{c}^{\textup{St}}$ we construct above is an $\pmb{\al}$-restricted Euler system.
 \item Consider the even, non-trivial character $\omega\chi^{-1}:=\psi$ of $G_k$, and set $T^\prime:=\ZZ_p(1)\otimes\psi^{-1}$. Only in this example, we let $\all$ denote a fixed $\ZZ_p[[\Gamma\times\pmb{\Delta}]]$-rank-one direct summand of $\varprojlim_{n,\tau}H^1(k_n(\tau),T^\prime)$. The author~\cite[\S3]{kbbiwasawa} has constructed an $\all$-restricted Euler system for the pair $(T^\prime,\kk)$, starting from the conjectural Rubin-Stark elements.

 Later in \S\ref{subsec:twiststark}, we will construct another $\all$-restricted Euler system for $(T^\prime,\kk)$, applying a formal twisting argument on the Euler system $\textbf{c}^{\textup{St}}$. We will also compare this Euler system with the one coming from the Rubin-Stark elements,  using the ``rigidity" of the collection of $\LL$-adic Kolyvagin systems.
 \end{enumerate}
  \end{example}

 \subsection{Kolyvagin systems for modified Selmer groups (bis)}
  \label{subsec:kolsys2}
 Recall the sets $\PP_j\subset \PP$ and $\NN_j\subset \NN$. For notational simplicity, we write $\mathbb{T}:=T\otimes\LL$ from now on, and for a fixed topological generator $\gamma\in\Gamma$, we set $\gamma_n=\gamma^{p^n}$. Finally, let $\mathcal{M}$ be the maximal ideal of the ring $\LL$.

 \begin{define}
\label{def:alternativeKS}
For $\mathbb{F}=\FF_{\LL}$ or $\FF_{\all}$, we set $$\overline{\KS}^{\prime}(\mathbb{T},\mathbb{F},\PP):=\varprojlim_{m,n}\varinjlim_j \textup{\KS}(\mathbb{T}/(p^m,\gamma_{n}-1)\mathbb{T},\mathbb{F}, \PP_{j}),$$
where $\textup{\KS}(\mathbb{T}/(p^m,\gamma_{n}-1)\mathbb{T},\mathbb{F}, \PP_{j})$ is the $\LL/(p^m,\gamma_n-1)$-module of Kolyvagin systems (in the sense of~\cite[Definition 3.1.3]{mr02}) for the propagated Selmer structure $\mathbb{F}$ on the quotient $\mathbb{T}/(p^m,\gamma_{n}-1)\mathbb{T}$.
\end{define}
\begin{rem}
\label{1-kolsys} We introduced the module $\overline{\KS}^{\prime}(\mathbb{T},\FF_{\LL},\PP)$ above because, after applying Kolyvagin's descent procedure~\cite[\S IV]{r00}, one directly obtains elements of $\overline{\KS}^{\prime}(\mathbb{T},\FF_{\LL},\PP)$. On the other hand, it is not hard to see for $\mathbb{F}=\FF_{\LL}$ or $\FF_{\all}$ that the module  $\overline{\KS}^{\prime}(\mathbb{T},\mathbb{F},\PP)$ defined above is naturally isomorphic to the module $\overline{\KS}(\mathbb{T},\mathbb{F},\PP)$ of Definition~\ref{def:KSboth}, using the fact that each of the collections $\{p^m,\gamma_n-1\}_{_{m,n}}$ and $\{p^m,\xx^n\}_{_{m,n}}$ forms a base of neighborhoods at $zero$.  Furthermore, using the fact that the collection $\{\mathcal{M}^\alpha\}_{\alpha\in \ZZ^+}$ also forms a base of neighborhoods at $zero$, one may identify these two modules Kolyvagin systems by the generalized module of Kolyvagin systems defined in~\cite[Definition 3.1.6]{mr02}. By slight abuse,  we will write  $\overline{\KS}(\mathbb{T},\mathbb{F},\PP)$ for any of the three modules of Kolyvagin systems given by three different definitions (i.e., by Definition~\ref{def:KSboth} and \ref{def:alternativeKS} here; and~\cite[Definition 3.1.6]{mr02}). For our purposes in this section, we will use Definition~\ref{def:alternativeKS} to define this module.
 \end{rem}

 Write $$\left\{\{\kappa^{\textup{St}}_{\tau,m}\}_{_{\tau \in \NN_m}}\right\}_m=\pmb{\kappa}^{\textup{St}} \in \overline{\KS}(T,\FFc,\PP)$$ (resp.,
 $$\left\{\left\{\kappa^{\textup{St}_\infty}_{\tau}(m,n)\right\}_{\tau \in \NN_{m+n}}\right\}_{m,n}=\pmb{\kappa}^{\textup{St}_\infty} \in \overline{\KS}(\mathbb{T},\FF_{\LL},\PP)$$
 for the Kolyvagin systems obtained via the descent procedure of~\cite[\S4]{r00} applied on the Euler system $\textbf{c}^{\textup{St}}=\{c^{\textup{St}}_{k_n(\tau)}\}$. We know that
$$
\xymatrix@R=.2cm @C=.1cm{
{\kappa}^{\textup{St}}_1\ar@{=}[d]  & \ar@{=}[l]_(.83){\textup{def}} \varprojlim_{m} \kappa_{1,m}^{\textup{St}} \in \varprojlim_{m}H^1(k,{T}/p^m{T}) = H^1(k,{T}) \\
c_k^{\textup{St}}&\stackrel{\textup{def}}{=}\lambda\circ
\theta_L^{\chi}=\varphi\circ \theta_L^{\chi}\in \Hom
\left(\left(\mathbb{A}_L^{\times}/L^\times\right)^{\chi},\ZZ_p\right)=H^1(k,T)
}$$
 and
 $$\xymatrix@R=.2cm @C=.2cm{ {\kappa}^{\textup{St}_\infty}_1\ar@{=}[d]  & \ar@{=}[l]_(.89){\textup{def}} \varprojlim_{m,n} \kappa_{1}(m,n) \in \varprojlim_{m,n}H^1(k,\mathbb{T}/(p^m,\gamma_n-1)\mathbb{T})=H^1(k,\mathbb{T})\\
 \{c_{k_n}^{\textup{St}}\}_{_n}&\ar@{=}[l]_(.85){\textup{def}}\{\lambda_n\circ \theta_{L_n}^{\chi}\}_{_n}=\{\varphi_n\circ \theta_{L_n}^{\chi}\}_{_n} \in \varprojlim_n H^1(k_n,T)=H^1(k,\mathbb{T}).
}$$
\begin{rem}
\label{shapiro}
For every (rational) prime $\ell$, Shapiro's lemma shows that
\be\label{eqn:shapiroglobal} H^1(k(\tau),\mathbb{T}/(p^m,\gamma_{n}-1)\mathbb{T}) \cong H^1(k_n(\tau),T/p^mT) \hbox{ and,}\ee
\be\label{eqn:shapirolocal} H^1(k(\tau)_\ell,\mathbb{T}/(p^m,\gamma_{n}-1)\mathbb{T}) \cong H^1(k_n(\tau)_\ell,T/p^mT).\ee
See~\cite[Proposition II.1.1]{colmez-reciprocity} for (\ref{eqn:shapiroglobal}) and~\cite[Appendix B.5]{r00}  for (\ref{eqn:shapirolocal}). Thanks to these identifications, we may talk about the propagation of a local condition $H^1_{\FF}(k_\ell,\mathbb{T})\subset H^1(k_\ell,\mathbb{T})$ at $\ell$ to a local condition
$$H^1_{\FF}((k_n)_\ell,T/p^mT)\subset H^1((k_n)_\ell,T/p^mT)\cong H^1(k_\ell,\mathbb{T}/(p^m,\gamma_n-1)\mathbb{T}).$$ Namely, we define $H^1_{\FF}((k_n)_\ell,T/p^mT)$ as the isomorphic copy of the module $H^1_{\FF}(k_\ell,\mathbb{T}/(p^m,\gamma_n-1)\mathbb{T})$ under the isomorphism (\ref{eqn:shapirolocal}) of Shapiro's lemma.
\end{rem}

 \begin{thm}
\label{thm:kolsys2}
\begin{enumerate} \item[(i)]$\pmb{\kappa}^{\textup{St}} \in \overline{\KS}(T,\FF_{\al},\PP).$
\item[(ii)] $\pmb{\kappa}^{\textup{St}_\infty} \in \overline{\KS}(\mathbb{T},\FF_{\all},\PP).$
\end{enumerate}
\end{thm}
\begin{proof}
Identical to the proofs of~\cite[Theorem 2.19]{kbbstark} and~\cite[Theorem 3.23]{kbbiwasawa}. We remark that the only essential point beyond~\cite{r00,mr02} is to verify that
\be\label{alpha}
\textup{loc}_p\left( \kappa^{\textup{St}}_{\tau,m}\right) \in H^1_{\FF_{\al}}(k_p,T/p^mT)\cong\al/p^m\al
\ee for each $\tau\in \NN_m$ and $m\in \ZZ^+$; and that
\begin{align}
\label{beta}\textup{loc}_p\left( \kappa^{\textup{St}_\infty}_{\tau}(m,n)\right) &\in H^1_{\FF_{\all}}(k_p, \mathbb{T}/(p^m,\gamma_n-1)\mathbb{T})\\
&\cong \all/(p^m,\gamma_n-1)\all :=\al_n/p^m\al_n
\end{align}
for every $\tau \in \NN_{m+n}$ and every $m,n \in \ZZ^+$. As in~\cite{kbbstark,kbbiwasawa}, the key point in proving the assertions (\ref{alpha}) and (\ref{beta}) is the fact that $\textbf{c}^{\textup{St}}$ is an $\pmb{\al}$-restricted Euler system.
\end{proof}

We give the main applications of our construction in Section~\ref{sec:applications}. This will be twofold: The first application is somewhat standard; we will bound the dual Selmer groups. As the second application, we will relate the Stickelberger elements, making use of the first application, to the conjectural Rubin-Stark elements. Among other things, this will enable us to control the local behavior of  Rubin-Stark elements.

\begin{rem}
As remarked earlier, one only needs a  \emph{weak Kolyvagin system}  in order to deduce the main applications of the $\pmb{\al}$-restricted Euler system $\textbf{c}^{\textup{St}}$. See~\cite[Definition 3.1.8]{mr02} for a definition of a weak Kolyvagin system. We remark that Kolyvagin's descent~\cite[\S4]{r00} applied on an Euler system gives rise to a weak Kolyvagin system. A weak Kolyvagin system can be used following the formalism of~\cite[\S5 and \S7]{r00} with slight alterations, to obtain the same results which we present below.
\end{rem}

 \section{Applications}
 \label{sec:applications}
 Before we state our main applications of the $\pmb{\al}$-restricted Euler system $\mathbf{c}^{\textup{St}}$, we recall our running hypotheses. We fix a totally odd character $\chi$ of $G_k:=\Gal(\overline{k}/k)$ which is \emph{not} the Teichm\"uller character $\omega$  (giving the action of $G_k$ on the $p$-th roots of unity $\pmb{\mu}_{p}$).  Throughout Section~\ref{sec:applications}, we assume that (A1) holds. Suppose also that Assumption~\ref{assume:brumer} is true.
 \subsection{Main theorem over $k$}
We first prove a bound on the size of the dual Selmer group $H^1_{\FF_{\al}^*}(k,T^*)$. We use this bound to obtain bounds on the classical (dual) Selmer groups, via the comparison Theorem established in~\S\ref{subsubsec:compareselmer1}.

\begin{thm}
\label{thm:mainmodifiedk-1}Under our running hypotheses,
\begin{enumerate}
\item[(i)] $\textup{length}_{\ZZ_p} (H^1_{\FF^*_{\al}}(k,T^*)) \leq \textup{length}_{\ZZ_p} (H^1_{\FF_{\al}}(k,T)/\ZZ_p\cdot\kappa_1^{\textup{St}}),$
\item[(ii)] $\textup{length}_{\ZZ_p} (H^1_{\FF^*_{\textup{cl}}}(k,T^*)) \leq \textup{length}_{\ZZ_p} (\al/\ZZ_p\cdot c_k^{\textup{St}})$.
\end{enumerate}
\end{thm}
\begin{proof}
(i) is Theorem~\ref{thm:mainksapplicationk}. (ii) follows from (i) and Corollary~\ref{cor:compare over k with class} applied with $c=c_k^{\textup{St}}=\kappa_1^{\textup{St}}$.

\end{proof}
Let $\theta_L^\chi \in \ZZ_p[\Delta]^{\chi}$ be as in \S\ref{sec:ES}. The evaluation map
$$\chi: \ZZ_p[\Delta]^{\chi} \lra \ZZ_p$$
 induces an isomorphism and we write $\chi(\theta_L)$ for the image of $\theta_L^\chi$ under this map. Recall the definition of $\varphi$ and $\lambda$, which we used in \S\ref{sec:ES} and \S\ref{sec:ESKSmap} to define $\textbf{c}^{\textup{St}}_k$. Recall also that $\lambda=\varphi$ by definition.
\begin{thm}
\label{thm:maink}
Under the assumptions above,
 $$|A_L^\chi|\leq |\ZZ_p/\chi(\theta_L)\ZZ_p|.$$
\end{thm}

\begin{proof}
By Proposition~\ref{prop:classical selmer explicit}, $H^1_{\FF^*_{\textup{cl}}}(k,T^*)\cong A_L^\chi$, and by construction, $c_k^{\textup{St}}=\chi(\theta_L) \lambda=\chi(\theta_L)\varphi$. Since $\varphi$  is a $\ZZ_p$-generator of $\al$ (by definition), it follows that $\al/\ZZ_p\cdot c_k^{\textup{St}}=\ZZ_p/\chi(\theta_L)\ZZ_p$. The proof now follows from Theorem~\ref{thm:mainmodifiedk-1}(ii).
\end{proof}

The inequality of Theorem~\ref{thm:maink} may be strengthened to an equality:
\begin{thm}
\label{thm:mainequalityk} As in Theorem~\ref{thm:maink} above, assume that (A1) holds and Assumption 1.1 is true. Suppose in addition
\begin{itemize}
\item[(i)] either that $\pmb{\mu}_p \not\subset L$, or
\item[(ii)] the statement of Theorem~\ref{thm:maink} is true for $\chi=\omega$.
\end{itemize}
Then
$$|A_L^\chi|= |\ZZ_p/\chi(\theta_L)\ZZ_p|.$$
\end{thm}
\begin{proof}
Unless $\pmb{\mu}_p\subset L$ (i.e., if we are in the case (i) above), the claimed equality follows from the inequality of Theorem~\ref{thm:maink} using a standard argument involving the class number formula; see~\cite[\S5]{ru92} and \cite[\S3]{kbbstark} for details. Note that we need the assumption $\pmb{\mu}_p\not\subset L$ for this portion since otherwise, we would need the inequality of Theorem~\ref{thm:maink} also for the Teichm\"uller character $\omega$, and this escapes the methods of the current paper.

When $\pmb{\mu}_p\subset L$, Theorem~\ref{thm:maink} used along with our assumption (ii) gives again the desired equality utilizing the class number formula.
\end{proof}
We note that Wiles' result~\cite[Theorem 3]{wiles90} (\emph{only} for the case $\chi=\omega$) is exactly the statement of Theorem~\ref{thm:maink} in the case $\chi=\omega$. Therefore, the condition (ii) that appears in the statement of Theorem~\ref{thm:mainequalityk} is true if we assume Wiles' result in this particular case.

We henceforth assume \cite[Theorem 3]{wiles90} \emph{({only} for the case $\chi=\omega$)} if $\pmb{\mu}_p\subset L$.
\begin{rem}
Note that  all  the Euler factors at primes $\wp | p$ are excluded in the definition of $\theta_K^\chi$, for any $K\in \kk$ (recall the collection $\kk$ from~\S\ref{sec:ES}), contrary to the standard definition of Stickelberger elements when $K/k$ is unramified at a certain prime above $p$. Theorem~\ref{thm:mainequalityk} is still equivalent to \cite[Theorem 3]{wiles90}, since our assumption (A1) assures that our Stickelberger elements agree with that of~\cite{wiles90, kurihara} up to units.
\end{rem}
\subsection{Main theorem over $k_\infty$}
\label{subsec:mainthmkinfty}
Along with the assumptions above, suppose also that (A2) holds. Write $\textup{char}(M)$ for the characteristic ideal of a torsion $\LL$-module $M$.

We proceed as in the previous section: First, we prove a bound for the characteristic ideal of the dual Selmer group $H^1_{\FF_{\all}^*}(k,\mathbb{T}^*)^{\vee}$. We then use this bound, together with Proposition~\ref{prop:compare selmer over k_infty}, to obtain a bound on the characteristic ideal of (the Pontryagin dual of) the classical (dual) Selmer groups.

\begin{thm}
\label{thm:mainmodifiedkinfty-2} Under the running assumptions,
$$\textup{char}\left(H^1_{\FF_{\all}^*}(k,\mathbb{T}^*)^{\vee}\right)\mid \textup{char}\left(H^1_{\FF_{\all}}(k,\mathbb{T})/\Lambda\cdot\kappa_1^{\textup{St}_\infty}\right).$$
\end{thm}
\begin{proof}
This is Theorem~\ref{thm:mainapplicationKS}.
\end{proof}
Set $c_{k_\infty}^{\textup{St}}:=\{c_{k_n}^{\textup{St}}\}_{_n} \in \varprojlim_n H^1(k_n,T)=H^1(k,\mathbb{T})$.
\begin{cor}
\label{cor:mainkinfty-2}
$\textup{char}\left((\varinjlim_n A_{L_n}^\chi)^{\vee}\right)\mid \textup{char}\left(\all/\Lambda\cdot c_{k_\infty}^{\textup{St}_\infty}\right).$
\end{cor}
\begin{proof}
This follows from Theorem~\ref{thm:mainmodifiedkinfty-2} and Proposition~\ref{prop:compare selmer over k_infty}(ii) applied with $c=c_{k_\infty}^{\textup{St}_\infty}$; together with the fact that $\kappa_1^{\textup{St}_\infty}=c_{k_\infty}^{\textup{St}_\infty}$.
\end{proof}

Recall the element $\theta_{L_n}^\chi \in \ZZ_p[\Delta\times\Gamma_n]^{\chi}=\ZZ_p[\Delta]^{\chi}[\Gamma_n]$ from~\S\ref{sec:ES}. We denote the image of $\theta_{L_n}^\chi$ under the map $$\chi_{_{\LL}}:\ZZ_p[\Delta]^{\chi}[\Gamma_n] \lra \ZZ_p[\Gamma_n],$$ (which extends $\chi:\ZZ_p[\Delta]^\chi\ra \ZZ_p$ from the previous section  to $\Gamma_n$ by letting $\chi_{_{\LL}}(\gamma)=\gamma$ for $\gamma \in \Gamma_n$) by $\chi_{_{\LL}}(\theta_{L_n})$. Lemma~\ref{lem:stickdistribution} shows that $\{\chi_{_{\LL}}(\theta_{L_n})\}$ is a projective system with respect to natural surjections $\ZZ_p[\Gamma_{n^\prime}]\ra \ZZ_p[\Gamma_n]$, $n^\prime \geq n$. We define
$$\chi_{_{\LL}}(\Theta_{L_{\infty}}):=\{\chi_{_{\LL}}(\theta_{L_n})\} \in \varprojlim_n \ZZ_p[\Gamma_n]=\Lambda.$$
Finally, let $x \mapsto x^{\bullet}$ be the involution on $\LL$ induced from $\gamma \mapsto \gamma^{-1}$ for $\gamma \in \Gamma$.
\begin{thm}
\label{thm:mainkinfty}
Under the running hypotheses of this section,
$$\textup{char}\left((\varinjlim_n A_{L_n}^\chi)^{\vee}\right)\mid \chi_{_{\LL}}(\Theta_{L_{\infty}})^{\bullet}.$$
\end{thm}
\begin{proof}
By the construction of $c_{k_n}^{\textup{St}}$ and $\lambda_n$,
  $$c_{k_n}^{\textup{St}}=\lambda_n\circ \theta_{L_n}^{\chi}=\chi_{_{\LL}}(\theta_{L_{n}})^\bullet\lambda_n=\chi_{_{\LL}}(\theta_{L_{n}})^\bullet\varphi_n.$$
   It follows that
   \be
   \label{eqn:ESinftyexplicit}
   c_{k_\infty}^{\textup{St}}=\{\chi_{_{\LL}}(\theta_{L_{n}})^\bullet\varphi_n\}_n=\chi_{_{\LL}}(\Theta_{L_{\infty}})^{\bullet}\Phi,
   \ee
   with $\Phi=\{\varphi_n\}$ as in \S\ref{subsec:choosehoms}. Since $\Phi$ is a generator of $\all$ (by definition), Theorem follows from Corollary~\ref{cor:mainkinfty-2}.
\end{proof}

Once again, making use of a standard class number argument (and yet again the case $\pmb{\mu}_p\subset L$ requires more care as above) shows that this equality may be turned into an equality:
\begin{thm}
\label{thm:mainequalitykinfty} Under the hypotheses of Theorem~\ref{thm:mainequalityk} and assuming (A2),
$$\textup{char}\left((\varinjlim_n A_{L_n}^\chi)^{\vee}\right)= \chi_{_{\LL}}(\Theta_{L_{\infty}})^\bullet.$$
\end{thm}

Let $\rho_{\textup{cyc}}:G_k\ra\ZZ_p^\times$ be the cyclotomic character (giving the action of $G_k$ on $\pmb{\mu}_{p^{\infty}}$), and set $\langle\rho_{\textup{cyc}}\rangle=\omega^{-1}\rho_{\textup{cyc}} :\Gamma \ra 1+p\ZZ_p$. We define a twisting map $\textup{Tw}_{\langle\rho_{\textup{cyc}}\rangle}:\LL\ra\LL$ by setting
$$\textup{Tw}_{\langle\rho_{\textup{cyc}}\rangle}(\gamma)= \langle\rho_{\textup{cyc}}\rangle(\gamma)\gamma \hbox{ for } \gamma \in \Gamma,$$
 and extending to $\LL$ by linearity and continuity.  Finally, let $\al_{\omega\chi^{-1}} \in \LL$ denote the \emph{Deligne-Ribet $p$-adic $L$-function} for the character $\omega\chi^{-1}$. We will loosely say here that $\al_{\omega\chi^{-1}}$ is characterized by the following interpolation property:
\be\label{eqn:interpolation}\langle\rho_{\textup{cyc}}\rangle^k\xi(\al_{\omega\chi^{-1}})=\prod_{\wp|p}(1-\omega^{-k}\xi\chi(\wp)\mathbf{N}\wp^{k-1})L(1-k,\omega^{-k}\xi\chi),\ee
for every $k\geq 1$ and every character $\xi$ of $\Gamma$ of finite order. Here, $L(s,\varrho)$ is the (abelian) Artin $L$-function attached to a  character $\varrho$ of $G_k$ which is of finite order.

\begin{lemma}
\label{lem:twistofpadicL}
 $\chi_{_{\LL}}(\Theta_{L_{\infty}})^\bullet=\textup{Tw}_{\langle\rho_{\textup{cyc}}\rangle}(\al_{\omega\chi^{-1}})$.
\end{lemma}

\begin{proof}
For every character $\xi$ of $\Gamma$ of finite order, it follows from the definitions that
\begin{align*}
\xi(\chi_{_{\LL}}(\Theta_{L_{\infty}})^\bullet)=\xi^{-1}\left(\chi_{_{\LL}}(\Theta_{L_{\infty}})\right)&=\prod_{\wp|p}(1-\chi^{-1}\xi(\wp))L(0,\chi^{-1}\xi)\\
 &=\langle\rho_{\textup{cyc}}\rangle\xi(\al_{\omega\chi^{-1}})= \xi(\langle\rho_{\textup{cyc}}\rangle\al_{\omega\chi^{-1}}).
\end{align*}
Since this is true for every $\xi$, Lemma follows.
\end{proof}

\begin{cor}
\label{cor:mainequalityinfty}
Under the assumptions of Theorem~\ref{thm:mainequalitykinfty},
$$\textup{char}\left((\varinjlim_n A_{L_n}^\chi)^{\vee}\right)=\textup{Tw}_{\langle\rho_{\textup{cyc}}\rangle}(\al_{\omega\chi^{-1}}).$$
\end{cor}

\subsection{Twisting and local Iwasawa theory of Stark elements}
\label{subsec:twiststark}
The goal of this section is to establish a connection between the Stickelberger elements and the conjectural Rubin-Stark elements. To achieve this, we will employ the following two ingredients:
\begin{enumerate}
\item The twisting formalism developed in~\cite[\S VI]{r00}.
\item The rigidity statement~\cite[Theorem 2.19(ii)]{kbbiwasawa} for the module of $\LL$-adic Kolyvagin systems.
\end{enumerate}
Write $\rho:= \langle \rho_{\textup{cyc}} \rangle$ for notational simplicity. Let $\psi=\omega\chi^{-1}$ as in the introduction and set $$T^\prime=\ZZ_p(1)\otimes\psi^{-1}=T\otimes\rho.$$
We also write $\mathbb{T}^\prime=T^\prime\otimes\LL$.
\subsubsection{Twisting argument}
\label{subsub:twistingargument}
\begin{lemma}
\label{lem:rubintwisiting}
Suppose $\rho$ is as above.
\begin{enumerate}
\item[(i)] There is a commutative diagram
$$
\xymatrix{
\varprojlim_{n} H^1(k_n,T)\otimes\rho \ar[r]^(.52){\sim}\ar[d]_{\textup{loc}_p}& \varprojlim_{n} H^1(k_n,T^\prime)\ar[d]^{\textup{loc}_p} \\
\varprojlim_{n} H^1((k_n)_p,T)\otimes\rho \ar[r]^(.52){\sim} &\varprojlim_{n} H^1((k_n)_p,T^\prime)
}
$$
such that the horizontal arrows are natural isomorphisms.
\item[(ii)] There is a commutative diagram
$$
\xymatrix{
\varprojlim_{K\in\kk_0} H^1(K,T)\otimes\rho \ar[r]^(.52){\sim}\ar[d]_{\textup{loc}_p}& \varprojlim_{K\in\kk_0} H^1(K,T^\prime)\ar[d]^{\textup{loc}_p} \\
\varprojlim_{K\in\kk_0} H^1(K_p,T)\otimes\rho \ar[r]^(.52){\sim} &\varprojlim_{K\in\kk_0} H^1(K_p,T^\prime)
}
$$
such that the horizontal arrows are natural isomorphisms,
\item[(iii)] $\left(\varinjlim_n H^1_{\FF_{\cl}^*}(k_n,{T}^*)\right)^\vee \otimes\rho \stackrel{\sim}{\lra} \left(\varinjlim_n H^1_{\FF_{\cl}^*}(k_n,(T^\prime)^*)\right)^\vee$.
\end{enumerate}
\end{lemma}
\begin{proof}
This is~\cite[Proposition VI.2.1]{r00}. We note that, $\varinjlim_n H^1_{\FF_{\cl}^*}(k_n,{T}^*)$ here coincides with what Rubin calls $\mathcal{S}_{\Sigma_p}(k_\infty,W)$, thanks to Remark~\ref{rem:compare classical with can}, where $W=T^*$ and $\Sigma_p$ is the set of places of $k$ which lie above $p$.
\end{proof}
Let $\pmb{\al} \subset\varprojlim_{K\in \kk_0} H^1(K_p,T)$ be as in \S\ref{def:thick line}, and let $\pmb{\al}_{\rho}$ denote the image of $\pmb{\al}$ under the isomorphism of Lemma~\ref{lem:rubintwisiting}(ii). Recall the $\pmb{\al}$-restricted Euler system $\mathbf{c}^{\textup{St}}$ that we constructed at the end of~\S\ref{subsec:choosehoms}.
Let $\mathbf{c}^{\textup{St},\rho}$ denote the twist of the Euler system $\mathbf{c}^{\textup{St}}$ defined via \cite[Theorem VI.3.5]{r00}; this means $\mathbf{c}^{\textup{St},\rho} \in \ES(T^\prime,\kk)$. Furthermore, one can see without difficulty that the following Lemma is true:
\begin{lemma}
\label{lem:twistedESrestricted}
The twisted Euler system $\mathbf{c}^{\textup{St},\rho}$ is $\pmb{\al}_\rho$-restricted.
\end{lemma}

Let $\all_\rho \in H^1(k_p,\mathbb{T}^\prime)$ denote the image of $\pmb{\al}_\rho$ under the obvious projection map. Recall the element ${c}^{\textup{St}}_{k_\infty}=\{c_{k_n}^{\textup{St}}\}_n \in H^1(k,\mathbb{T})$, and set $c^{\textup{St},\rho}_{k_\infty}=\{c_{k_n}^{\textup{St},\rho}\}_n \in H^1(k,\mathbb{T}^\prime)$. Note that $\textup{loc}_p({c}^{\textup{St}}_{k_\infty}) \in \all \subset H^1(k_p,\mathbb{T})$ and $\textup{loc}_p({c}^{\textup{St},\rho}_{k_\infty}) \in \all_\rho \subset H^1(k_p,\mathbb{T}^\prime)$, by construction. As before, we drop $\textup{loc}_p$ from notation and we denote $\textup{loc}_p({c}^{\textup{St}}_{k_\infty})$ (resp., $\textup{loc}_p({c}^{\textup{St},\rho}_{k_\infty})$) simply by ${c}^{\textup{St}}_{k_\infty}$ (resp., by ${c}^{\textup{St},\rho}_{k_\infty}$). Lemma~\ref{lem:rubintwisiting}(i) induces an isomorphism
\be
\label{eqn:twistedlinesisom}
\all/\LL\cdot {c}^{\textup{St}}_{k_\infty} \otimes\rho \stackrel{\sim}{\lra} \all_\rho/\LL\cdot {c}^{\textup{St},\rho}_{k_\infty}.
\ee
To simplify notation, set $X_\infty(T):=\left(\varinjlim_n H^1_{\FF_{\cl}^*}(k_n,{T}^*)\right)^\vee$ and set similarly $X_\infty(T^\prime):=\left(\varinjlim_n H^1_{\FF_{\cl}^*}(k_n,{(T^\prime)}^*)\right)^\vee$. Until the end of this paper, assume (A1) and (A2) both hold true.
\begin{prop}
\label{prop:charidealcomparison}
Let $\textup{Tw}_\rho:\LL\ra\LL$ be the twisting operator as above.
\begin{enumerate}
\item[(i)]$\textup{Tw}_\rho\left(\textup{char}\left( X_\infty(T^\prime) \right)\right)=\textup{char}\left(X_\infty(T)\right).$
\item[(ii)] $\textup{Tw}_\rho\left(\textup{char}\left(\all_\rho/\LL\cdot {c}^{\textup{St},\rho}_{k_\infty}\right)\right)=\textup{char}\left( \all/\LL\cdot {c}^{\textup{St}}_{k_\infty}\right).$
\item[(iii)] $\textup{char}\left(X_\infty(T^\prime)\right)=\textup{char}\left( \all_\rho/\LL\cdot {c}^{\textup{St},\rho}_{k_\infty}\right).$
\end{enumerate}
\end{prop}

\begin{proof}
The assertion (i) follows from Lemma~\ref{lem:rubintwisiting}(iii) together with \cite[Lemma VI.1.2(i)]{r00}, and (ii) from (\ref{eqn:twistedlinesisom}) applied with \cite[Lemma VI.1.2(i)]{r00}.

Thanks to (i) and (ii), the assertion (iii) is equivalent to verifying that
\be
\label{eqn:twistedmainconjstatement}
\textup{Tw}_{\rho^{-1}}\left(\textup{char}\left( X_\infty(T) \right)\right)=\textup{Tw}_{\rho^{-1}}\left(\textup{char}\left(\all/\LL\cdot {c}^{\textup{St}}_{k_\infty}\right)\right).
\ee
This, however, is the statement of Theorem~\ref{thm:mainequalitykinfty} put together with~(\ref{eqn:ESinftyexplicit}), and twisted by $\textup{Tw}_{\rho^{-1}}$.
\end{proof}
\subsubsection{Stickelberger elements vs. Rubin-Stark elements}
Following Definition 2.10 of \cite{kbbiwasawa}, we define the $\all_\rho$-modified Selmer structure $\FF_{\all_\rho}$ on $\mathbb{T}^{\prime}$, for $\all_{\rho}\subset H^1(k_p,\mathbb{T}^\prime)$ as above. The argument of~\cite[Theorem 3.23]{kbbiwasawa} shows that the $\pmb{\al}_\rho$-restricted Euler system $\mathbf{c}^{\textup{St},\rho}$ gives rise to a Kolyvagin system $\pmb{\kappa}^{\textup{St}_\infty,\rho} \in \overline{\KS}(\mathbb{T}^{\prime},\FF_{\all_\rho},\PP)$.
\label{subsub:comparestark}
\begin{prop}
\label{prop:applstikcrhomain} We have
$$\textup{char}\left(H^1_{\FF_{\all_\rho^*}}(k,(\mathbb{T}^{\prime})^*)^\vee\right)=\textup{char}\left(H^1_{\FF_{\all_\rho}}(k,\mathbb{T}^\prime)/\LL\cdot\kappa_1^{\textup{St}_\infty,\rho}\right).$$
\end{prop}
\begin{proof}
This follows from Proposition~\ref{prop:charidealcomparison} and the exact sequence of Proposition~\ref{prop:compare selmer over k_infty}(ii) rewritten for $\mathbb{T}^\prime$ instead of $\mathbb{T}$. See also~\cite[Proposition 2.12]{kbbiwasawa} for this version of the exact sequence of Proposition~\ref{prop:compare selmer over k_infty}(ii).
\end{proof}
Assume until the end of this paper that the character $\psi$ satisfies the hypothesis (A1) as well, that is, $\psi(\wp)\neq 1$ for any prime $\wp$ of $k$ lying above $p$. According to~\cite[Theorem 2.19(ii)]{kbbiwasawa}, the $\LL$-module of Kolyvagin systems $\overline{\KS}(\mathbb{T}^{\prime},\FF_{\all_\rho},\PP)$ is free of rank one.
\begin{cor}
\label{cor:KSrhoprimitive}
The Kolyvagin system $\pmb{\kappa}^{\textup{St}_\infty,\rho}$ generates the cyclic module $\overline{\KS}(\mathbb{T}^{\prime},\FF_{\all_\rho},\PP)$.
\end{cor}
\begin{proof}
We know by~\cite[Theorem 2.19(ii) and Proposition 4.2]{kbbiwasawa} that the free $\LL$-module $\overline{\KS}(\mathbb{T}^{\prime},\FF_{\all_\rho},\PP)$ is generated by a
$\LL$-primitive Kolyvagin system $\pmb{\kappa}$ (in the sense of~\cite[Definition 5.3.9]{mr02}). In particular, we may write $$\pmb{\kappa}^{\textup{St}_\infty,\rho}=u\cdot\pmb{\kappa},$$ with $u \in \LL$.

The main application of the $\LL$-primitive Kolyvagin system $\pmb{\kappa}$ for the Selmer triple $(\mathbb{T}^{\prime},\FF_{\all_\rho},\PP)$ is the following (see~\cite[Theorem 2.20]{kbbiwasawa}):
\be
\label{eqn:mainapplicationprimKS}
\textup{char}\left(H^1_{\FF_{\all_\rho^*}}(k,(\mathbb{T}^{\prime})^*)^\vee\right)=\textup{char}\left(H^1_{\FF_{\all_\rho}}(k,\mathbb{T}^\prime)/\LL\cdot\kappa_1\right).
\ee
The assertion (\ref{eqn:mainapplicationprimKS}) together with Proposition~\ref{prop:applstikcrhomain} shows that $u \in \LL^\times$. This completes the proof.
\end{proof}
Let $\pmb{\frak{e}}^{\textup{Stark}}=\{\frak{e}^{\textup{Stark}}\}_{K\in \kk_0}$ be the Euler system $\left\{\varepsilon_{K,\Phi_0^{(\infty)}}^{\psi}\right\}_{K\in\kk_0}$ defined in~\cite[Proposition 3.14]{kbbiwasawa}. This Euler system is obtained from the  Rubin-Stark elements that Rubin~\cite{ru96} conjectured to exist, and therefore the existence of $\pmb{\frak{e}}^{\textup{Stark}}$ is implicitly assumed\footnote{Having said that, note that we will recover in Theorem~\ref{thm:maincomparison}(i) below the Kolyvagin system $\pmb{\kappa}^{\textup{Stark}}$ up to a unit (which descends from $\pmb{\frak{e}}^{\textup{Stark}}$) directly from the Kolyvagin system $\pmb{\kappa}^{\textup{St}_\infty}$, which is obtained in this article from Stickelberger elements.} here. Note that our $\psi$ here is denoted by $\chi$ in \cite{kbbiwasawa}. Let
$$c^{\textup{Stark}}_{k_\infty} \in \wedge^r H^1(k_p,\mathbb{T}^{\prime})$$
 be the element defined as in~\cite[Remark 4.5]{kbbiwasawa}. Finally, let $\pmb{\kappa}^{\textup{Stark}}$ be the $\LL$-primitive Kolyvagin system for the Selmer triple $(\mathbb{T}^\prime,\FF_{\all_\rho},\PP)$, which is obtained from the Euler system $\pmb{\frak{e}}^{\textup{Stark}}$. We remark that what we call $\all_\rho$ here is denoted by $\all_\infty$ in~\cite{kbbiwasawa}, and $\pmb{\kappa}^{\textup{Stark}}$ here is denoted by $\pmb{\kappa}^{\Phi_0^{(\infty)}}$ in loc.cit. Recall the Kolyvagin system $\pmb{\kappa}^{\textup{St}_\infty,\rho}$ which is obtained from the $\pmb{\al}_{\rho}$-restricted Euler system $\textbf{c}^{\textup{St},\rho}$ of Stickelberger elements.

 The theorem below draws a connection between the Stickelberger elements and the Rubin-Stark elements.

\begin{thm}
\label{thm:maincomparison}
The following holds under the running assumptions:
\begin{itemize}
\item[(i)] There is a unit $u \in \LL^\times$ such that $\pmb{\kappa}^{\textup{St}_\infty,\rho}=u\cdot\pmb{\kappa}^{\textup{Stark}}$,
\item[(ii)] $\textup{char}\left(\all_\rho/\LL\cdot\kappa_1^{\textup{St}_\infty,\rho}\right)= \textup{char}\left(\all_\rho/\kappa_1^{\textup{Stark}}\right)$,
\item[(iii)] $\textup{char} \left(\wedge^r H^1(k_p,\mathbb{T}^\prime)/\LL\cdot c_{k_\infty}^{\textup{Stark}}\right)=\al_{\psi}$.
\end{itemize}
Here, $\al_{\psi}$ is the Deligne-Ribet $p$-adic $L$-function attached to the totally even character $\psi$.
\end{thm}
\begin{proof}
The assertion (i) follows from Corollary~\ref{cor:KSrhoprimitive} and the fact that $\pmb{\kappa}^{\textup{Stark}}$ generates the free rank-one $\LL$-module $\overline{\KS}(\mathbb{T}^\prime,\FF_{\all_\rho},\PP)$. The assertion (ii) is immediate from (i).

The discussion preceding the statement of~\cite[Corollary 4.6]{kbbiwasawa} shows that
$$\wedge^r H^1(k_p,\mathbb{T}^\prime)/\LL\cdot c_{k_\infty}^{\textup{Stark}} \cong \all_\rho/\LL\cdot \kappa_1^{\textup{Stark}}.$$
Hence, it follows from (ii) that
$$\textup{char} \left(\wedge^r H^1(k_p,\mathbb{T}^\prime)/\LL\cdot c_{k_\infty}^{\textup{Stark}}\right)=\textup{char}\left(\all_\rho/\LL\cdot\kappa_1^{\textup{St}_\infty,\rho}\right).$$
By the construction of the Kolyvagin system $\pmb{\kappa}^{\textup{St}_{\infty},\rho}$  out of the Euler system $\textbf{c}^{\textup{St},\rho}$, it follows that $\kappa_1^{\textup{St}_\infty,\rho}=c_{k_\infty}^{\textup{St},\rho}$, which in return implies that
$$\textup{char} \left(\wedge^r H^1(k_p,\mathbb{T}^\prime)/\LL\cdot c_{k_\infty}^{\textup{Stark}}\right)=\textup{char}\left(\all_\rho/\LL\cdot c_{k_\infty}^{\textup{St},\rho}\right).$$
Using Proposition~\ref{prop:charidealcomparison}(ii), Equation (\ref{eqn:ESinftyexplicit}) and Lemma~\ref{lem:twistofpadicL}, we see that $$\textup{Tw}_{\rho}\left(\textup{char}\left(\all_\rho/\LL\cdot c_{k_\infty}^{\textup{St},\rho}\right)\right)=\textup{char}\left(\all/\LL\cdot c_{k_\infty}^{\textup{St}}\right)=\textup{Tw}_{\rho}(\al_\psi).$$
This completes the proof of the theorem.
\end{proof}
The author~\cite[Theorem 4.7]{kbbiwasawa} has previously deduced Theorem~\ref{thm:maincomparison}(iii) from Iwasawa's main conjecture. Here, we need to assume slightly less, namely, the $\chi$-part of the Brumer's conjecture (Assumption~\ref{assume:brumer}) to prove this statement.
\\\\\textit{Acknowledgements}. {The author was supported by the William Hodge
post-doctoral fellowship of IH\'ES when this paper was written up.
The author wishes to extend his hearty thanks to IH\'ES for their warm
hospitality. He also thanks Masato Kurihara for his interest in our work and for an encouraging discussion on the results of this article when the author visited Keio; and the anonymous referee for numerous comments and suggestions to improve the exposition.
The author was partially supported by the Marie Curie grant EU-FP7-IRG 230668  the TUBITAK-Career grant 109T662.}

\bibliographystyle{halpha}
\bibliography{references}

\begin{thebibliography}{B{\"u}y10b}

\bibitem[B{\"u}y09a]{kbbstark}
K\^az{\i}m B{\"u}y\"ukboduk.
\newblock Kolyvagin systems of {S}tark units.
\newblock {\em J. Reine Angew. Math.}, 631:85--107, 2009.

\bibitem[B{\"u}y09b]{kbbiwasawa}
K\^az{\i}m B{\"u}y\"ukboduk.
\newblock Stark units and the main conjectures for totally real fields.
\newblock {\em Compositio Math.}, 145:1163--1195, 2009.

\bibitem[B{\"u}y10a]{kbb}
Kaz{\i}m B{\"u}y\"ukboduk.
\newblock {$\Lambda$}-adic {K}olyvagin {S}ystems, 2010.
\newblock IMRN, doi:10.1093/imrn/rnq186.

\bibitem[B{\"u}y10b]{kbbrankr}
K\^az{\i}m B{\"u}y\"ukboduk.
\newblock On {E}uler systems of rank $r$ and their {K}olyvagin systems.
\newblock {\em Indiana University Math. Journal}, 59(4):1277--1332, 2010.

\bibitem[Col98]{colmez-reciprocity}
Pierre Colmez.
\newblock Th\'eorie d'{I}wasawa des repr\'esentations de de {R}ham d'un corps
  local.
\newblock {\em Ann. of Math. (2)}, 148(2):485--571, 1998.

\bibitem[DR80]{deligne-ribet}
Pierre Deligne and Kenneth~A. Ribet.
\newblock Values of abelian {$L$}-functions at negative integers over totally
  real fields.
\newblock {\em Invent. Math.}, 59(3):227--286, 1980.

\bibitem[dS87]{deshalit}
Ehud de~Shalit.
\newblock {\em Iwasawa theory of elliptic curves with complex multiplication},
  volume~3 of {\em Perspectives in Mathematics}.
\newblock Academic Press Inc., Boston, MA, 1987.
\newblock $p$-adic $L$ functions.

\bibitem[Gre94]{gr2}
Ralph Greenberg.
\newblock Trivial zeros of {$p$}-adic {$L$}-functions.
\newblock In {\em $p$-adic monodromy and the Birch and Swinnerton-Dyer
  conjecture (Boston, MA, 1991)}, volume 165 of {\em Contemp. Math.}, pages
  149--174. Amer. Math. Soc., Providence, RI, 1994.

\bibitem[Gre04]{greither}
Cornelius Greither.
\newblock {Computing Fitting ideals of Iwasawa modules.}
\newblock {\em Math. Z.}, 246(4):733--767, 2004.

\bibitem[Kat99]{kato-es}
Kazuya Kato.
\newblock {Euler systems, Iwasawa theory, and Selmer groups.}
\newblock {\em Kodai Math. J.}, 22(3):313--372, 1999.

\bibitem[Kur03]{kurihara}
Masato Kurihara.
\newblock {On the structure of ideal class groups of CM-fields.}
\newblock {\em Doc. Math.}, J. DMV Extra Vol.:539--563, 2003.

\bibitem[Mil86]{milne}
J.~S. Milne.
\newblock {\em Arithmetic duality theorems}, volume~1 of {\em Perspectives in
  Mathematics}.
\newblock Academic Press Inc., Boston, MA, 1986.

\bibitem[MR04]{mr02}
Barry Mazur and Karl Rubin.
\newblock Kolyvagin systems.
\newblock {\em Mem. Amer. Math. Soc.}, 168(799):viii+96, 2004.

\bibitem[MR09]{mrdarmon}
Barry Mazur and Karl Rubin.
\newblock Refined class number formulas and {K}olyvagin systems, 2009.
\newblock arXiv 0909.3916v1.

\bibitem[PR94]{pr}
Bernadette Perrin-Riou.
\newblock Th\'eorie d'{I}wasawa des repr\'esentations {$p$}-adiques sur un
  corps local.
\newblock {\em Invent. Math.}, 115(1):81--161, 1994.
\newblock With an appendix by Jean-Marc Fontaine.

\bibitem[PR98]{pr-es}
Bernadette Perrin-Riou.
\newblock Syst\`emes d'{E}uler {$p$}-adiques et th\'eorie d'{I}wasawa.
\newblock {\em Ann. Inst. Fourier (Grenoble)}, 48(5):1231--1307, 1998.

\bibitem[Rub92]{ru92}
Karl Rubin.
\newblock Stark units and {K}olyvagin's ``{E}uler systems''.
\newblock {\em J. Reine Angew. Math.}, 425:141--154, 1992.

\bibitem[Rub96]{ru96}
Karl Rubin.
\newblock A {S}tark conjecture ``over {$\bold Z$}'' for abelian {$L$}-functions
  with multiple zeros.
\newblock {\em Ann. Inst. Fourier (Grenoble)}, 46(1):33--62, 1996.

\bibitem[Rub00]{r00}
Karl Rubin.
\newblock {\em Euler systems}, volume 147 of {\em Annals of Mathematics
  Studies}.
\newblock Princeton University Press, Princeton, NJ, 2000.
\newblock Hermann Weyl Lectures. The Institute for Advanced Study.

\bibitem[Sie70]{siegel70}
Carl~Ludwig Siegel.
\newblock \"{U}ber die {F}ourierschen {K}oeffizienten von {M}odulformen.
\newblock {\em Nachr. Akad. Wiss. G\"ottingen Math.-Phys. Kl. II}, 1970:15--56,
  1970.

\bibitem[Tat84]{tate}
John Tate.
\newblock {\em Les conjectures de {S}tark sur les fonctions {$L$} d'{A}rtin en
  {$s=0$}}, volume~47 of {\em Progress in Mathematics}.
\newblock Birkh\"auser Boston Inc., Boston, MA, 1984.
\newblock Lecture notes edited by Dominique Bernardi and Norbert Schappacher.

\bibitem[Wil90a]{wiles90}
A.~Wiles.
\newblock On a conjecture of {B}rumer.
\newblock {\em Ann. of Math. (2)}, 131(3):555--565, 1990.

\bibitem[Wil90b]{wiles-mainconj}
Andrew Wiles.
\newblock The {I}wasawa conjecture for totally real fields.
\newblock {\em Ann. of Math. (2)}, 131(3):493--540, 1990.

\end{thebibliography}
\end{document}